\documentclass[a4paper,reqno]{amsart}
\usepackage[initials]{amsrefs}
\usepackage{graphicx}
\usepackage{amsmath,amssymb,amsthm,mathrsfs}
\usepackage{hyperref}
\usepackage{placeins}
\hypersetup{hidelinks,colorlinks=true,linkcolor=blue,citecolor=blue}
\evensidemargin=\oddsidemargin
\numberwithin{equation}{section}

\newtheorem{theorem}{Theorem}[section]
\newtheorem{lemma}[theorem]{Lemma}
\newtheorem{proposition}[theorem]{Proposition}

\theoremstyle{remark}
\newtheorem{remark}[theorem]{Remark} 

\newcommand{\A}{\mathcal A}
\newcommand{\Hh}{\mathcal H}
\newcommand{\Pp}{\mathcal P}
\newcommand{\zA}{\zeta_{\A}}
\newcommand{\e}{\varepsilon}
\newcommand{\Res}{\operatorname{Res}}
\newcommand{\Sodd}{\mathcal S_{\mathrm{odd}}}
\newcommand{\Gc}{\mathcal G}
\newcommand{\Bns}{\mathcal E}
\newcommand{\Cns}{\mathcal K}

\begin{document}

\title[Period-three secondary term in the hyperelliptic ensemble]
{The Period-Three Secondary Term in the Mean Value of $L(\tfrac12,\chi_D)$ \\ in the Hyperelliptic Ensemble}

\author[H. Jung]{Hwanyup Jung}
\address{Department of Mathematics Education, Chungbuk National University, 1 Chungdae-ro, Seowon-gu, Cheongju 28644, Republic of Korea}
\email{hyjung@chungbuk.ac.kr}

\begin{abstract}
Let $q$ be an odd prime power.  We revisit Florea's asymptotic formula for the first
moment of quadratic Dirichlet $L$-functions over the odd-degree hyperelliptic ensemble
$\mathcal H_{2g+1}$, and show that its secondary term is not the complete one.  The
square-dual generating function has three double poles of the same modulus on the
secondary circle: the positive real one reproduces Florea's polynomial, while the two
nonreal conjugate poles contribute at the same order.  The complete secondary term is
$q^{2g/3}(\alpha_{g\bmod3}g+\beta_{g\bmod3})$, with real coefficients whose dependence
on the genus has minimal period exactly three, and on at least two of the three residue
classes the nonreal poles contribute a term of order $gq^{2g/3}$.
We also prove the formula for every odd prime power $q$, whereas Florea assumes
$q\equiv1\bmod4$ throughout; this is achieved by a phase-normalized quadratic Gauss sum
that restores the multiplicativity and the local evaluations her argument uses.  Exact
finite-ensemble moments for $q=3$ exhibit the three residue-class trends.
\end{abstract}
\maketitle

\section{Introduction}

Let $q$ be an odd prime power.  We study the first moment of quadratic Dirichlet
$L$-functions at the central point over the odd-degree hyperelliptic ensemble,
$$
M_1(g):=\sum_{D\in\Hh_{2g+1}}L\!\left(\tfrac12,\chi_D\right),
$$
as $g\to\infty$ with $q$ fixed.  Here $\Hh_{2g+1}$ is the set of monic square-free
polynomials of degree $2g+1$ in $\mathbb F_q[T]$, and $\chi_D$ is the quadratic character
attached to $D$ (see \S\ref{subsec:notation}).
Andrade and Keating determined the principal term in this moment~\cite{AK12}.
Florea subsequently combined the square-free sieve with Poisson summation and obtained a
secondary term of size $gq^{2g/3}$, together with an error
$O_{q,\e}\!\left(q^{g(1+\e)/2}\right)$ \cite{Fl17}.  Her secondary term is a single
polynomial in the genus: it equals $q^{(2g+1)/3}R(2g+1)$ with $R$ a linear polynomial,
and it arises from the positive real pole of the square-dual generating function.

The point of the present paper is that this is not the whole secondary contribution.
In the square-dual calculation the relevant denominator is $(1-q^4z^3)^2$, so the
secondary circle $|z|=q^{-4/3}$ carries three double poles $z_k=q^{-4/3}\omega^k$
$(k=0,1,2)$, where $\omega=e^{2\pi i/3}$.  They have the same modulus and are crossed
simultaneously by a circular contour deformation.  The positive real pole $z_0$
reproduces Florea's polynomial; the other two are complex conjugates, contribute at the
same order, and enter with the phases $\omega^{k(g+1)}$, so that their sum is real and
depends on $g\bmod3$.  The complete secondary term is therefore
$$
q^{2g/3}\bigl(\alpha_{g\bmod3}g+\beta_{g\bmod3}\bigr),
$$
and the dependence is genuine: the first Fourier coefficient of $m\mapsto\alpha_m$ is
nonzero, so the minimal period is exactly $3$.  On at least two of the three residue
classes modulo $3$
the two nonreal poles contribute a term of order $gq^{2g/3}$, which exceeds the error
term $O_{q,\e}(q^{g(1+\e)/2})$ and cannot be absorbed into it
(Proposition~\ref{prop:nonreal-lower-bound}).  The secondary term stated in
\cite{Fl17}*{Theorem~1.1} is thus the contribution of $z_0$ alone
(Remark~\ref{rem:Florea-identification}): the present formula reduces to Florea's when
$z_1$ and $z_2$ are discarded, so the result is a completion, rather than a different
normalization, of the secondary-term calculation.

The relation is in fact exact: Florea's polynomial is the average, over $m\in\{0,1,2\}$,
of the three quantities $q^{2g/3}(\alpha_mg+\beta_m)$, of which the complete secondary
term is the one with $m\equiv g\bmod3$; see \eqref{eq:mean-over-classes}.
The three need not be close to that average: for $q=3$ the coefficients
$\alpha_0,\alpha_1,\alpha_2$ are approximately $7.49$, $-7.99$ and $3.50$ times their
mean, and one of them has the opposite sign.
The two calculations are moreover shown to agree up to the point of divergence: the double contour
integral \eqref{eq:square-double-integral} reached here coincides, factor for factor,
with the corresponding display of \cite{Fl17}*{Section~6}, so that the entire divergence
is localized in the residue bookkeeping at the secondary circle.
Subsection~\ref{subsec:Florea-comparison} carries out the comparison.
The result of \cite{AK12} is not affected: the term isolated
here has size $gq^{2g/3}$, which lies well inside the error term there.

A second contribution is that Theorem~\ref{thm:main} is proved for every odd prime
power $q$, whereas \cite{Fl17}*{Theorem~1.1} assumes that $q$ is a prime with
$q\equiv1\bmod4$.  That congruence is retained throughout \cite{Fl17}, so the case
$q\equiv3\bmod4$ is not treated there, and the one remaining case discussed in
\cite{Fl17}*{Section~9} is explicitly left uncarried.  The phase-normalized quadratic
Gauss sum introduced in \S\ref{subsec:phase-poisson} removes the constant-field
phase once and for all, and thereby disposes of all cases uniformly.  This is not a
matter of convenience: for $q\equiv3\bmod4$ the unnormalized quadratic Gauss sums are
not multiplicative in the modulus, so that the bivariate Euler product on which the
whole square-dual calculation rests would fail, whereas the normalized sums are
multiplicative for every odd prime power $q$.

Write $\Pp$ for the set of monic irreducible polynomials in $\mathbb F_q[T]$ and put
$|P|=q^{\deg P}$, and let $\zA(s)=(1-q^{1-s})^{-1}$ be the zeta function of
$\mathbb F_q[T]$.  Define
$$
\mathscr P(s)=\prod_{P\in\Pp}\left(1-\frac1{|P|^s(|P|+1)}\right).
$$
For $m\in\{0,1,2\}$, the coefficients $\alpha_m$ and $\beta_m$ are defined explicitly in
\eqref{eq:alpha-periodic} and \eqref{eq:beta-periodic}, respectively; they depend only
on $q$.
All the quantities entering the main result are now defined, and we can state it.

\begin{theorem}\label{thm:main}
For every fixed odd prime power $q$, every $\e>0$ and every $g\ge1$, one has
\begin{align}\label{eq:main-theorem}
M_1(g)=\frac{\mathscr P(1)}{2\zA(2)}q^{2g+1}\left(2g+2+\frac4{\log q}\frac{\mathscr P'}{\mathscr P}(1)\right)
+q^{2g/3}\bigl(\alpha_{g\bmod3}g+\beta_{g\bmod3}\bigr)+O_{q,\e}\!\left(q^{g(1+\e)/2}\right).
\end{align}
The coefficients $\alpha_m$, $\beta_m$ are real.  
Extending the pairs $(\alpha_m,\beta_m)$ periodically to $m\in\mathbb Z$, 
the resulting sequence has minimal period exactly $3$.
\end{theorem}

The three terms in \eqref{eq:main-theorem} are the principal term determined in
\cite{AK12}, the complete secondary term, and the error.

Sections~\ref{sec:standard-reductions} and \ref{sec:principal-nonsquare} follow Florea's
argument, with the modifications needed for an arbitrary odd prime power $q$, and reduce
Theorem~\ref{thm:main} to the evaluation of the square-dual contribution.  That
evaluation, including all three secondary double poles and the exact period-three
dependence, is carried out in Section~\ref{sec:periodic-term}.
Section~\ref{sec:numerical} gives exact finite-ensemble evidence; it is not used in the
proof.

\section{Preliminaries and the decomposition of the moment}\label{sec:standard-reductions}

The aim of this section is to reduce $M_1(g)$ to the decomposition
\eqref{eq:moment-decomposition-short} into zero-frequency, square-dual and
nonsquare-dual contributions.  Two ingredients are needed: the square-free sieve, which
removes the square-free condition on $D$, and the Poisson summation formula, which
converts the resulting character sums into sums over a dual variable.  We establish the
latter in a phase-normalized form valid for every odd prime power $q$.

\subsection{Notation, the functional equation, and the square-free sieve}\label{subsec:notation}

Throughout the paper, $q$ is a fixed odd prime power, $\A=\mathbb F_q[T]$, and $g\ge1$ is
an integer; all implied constants are independent of $g$.  
We write $\A^+$ for the set of monic polynomials, $\Pp$ for the set of monic irreducible
polynomials, whose elements are called prime polynomials, and $\Hh$ for the set of monic
square-free polynomials in $\A$.  
For a nonzero polynomial $f$, put $d(f)=\deg f$ and $|f|=q^{d(f)}$.
We denote by $\phi(f)$ the polynomial Euler totient over $\A$. 

For a subset $\mathcal U\subseteq\A^+$ and an integer $n\ge0$, we use the following notation:
$\mathcal U_n=\{f\in\mathcal U:d(f)=n\}$, $\mathcal U_{\le n}=\{f\in\mathcal U:d(f)\le n\}$.
Whenever an index is negative, we use the convention
$\mathcal U_n=\mathcal U_{\le n}=\varnothing$.  This convention applies in
particular to all Poisson ranges below.

For $\operatorname{Re}(s)>1$, the zeta function of $\A$ is given by 
$$
\zA(s)=\sum_{f\in\A^+}|f|^{-s}=\prod_{P\in\Pp}(1-|P|^{-s})^{-1}.
$$
It is well known that $\zA(s)=(1-q^{1-s})^{-1}$.
In the $u$-variable we put $\mathcal Z_{\A}(u):=(1-qu)^{-1}$, so that
$\mathcal Z_{\A}(q^{-s})=\zA(s)$.

For $A\in\A$ and $B\in\A^+$, let $\bigl(\frac AB\bigr)$ denote the quadratic residue
symbol (Jacobi symbol), so that $\bigl(\frac AB\bigr)=0$ precisely when $(A,B)\ne1$.
For coprime $A,B\in\A^+$, the quadratic reciprocity law over $\A$ gives
$$
\left(\frac AB\right)=\eta_q^{d(A)d(B)}\left(\frac BA\right), 
$$
where $\eta_q=(-1)^{(q-1)/2}$. 

For a completely multiplicative function $\chi:\A^+\to\mathbb C$ with $|\chi(f)|\le1$,
its $L$-function is defined by
\begin{equation*} 
L(s,\chi):=\sum_{f\in\A^+}\chi(f)|f|^{-s}=\prod_{P\in\Pp}\left(1-\chi(P)|P|^{-s}\right)^{-1}, 
\quad (\operatorname{Re}(s)>1).
\end{equation*}
Upon writing $u=q^{-s}$, we write
\begin{equation*} 
\mathcal L(u,\chi)=L(s,\chi)=\sum_{n\ge0}\Biggl(\sum_{f\in\A_n^+}\chi(f)\Biggr)u^n.
\end{equation*}
For nonconstant $D\in\Hh$, put
$\chi_D(f)=\bigl(\frac Df\bigr)$; with this convention the square-free sieve below is
applied directly in the $D$-variable, without first interchanging numerator and
denominator by quadratic reciprocity.  
For $D\in\Hh_{2g+1}$ this is the finite-prime part of the
quadratic Hecke character of $\mathbb F_q(T)$ associated with the extension
$\mathbb F_q(T)(\sqrt D)$, or equivalently with the curve $y^2=D(T)$.
If $\eta_q=-1$, its conductor has an infinite component; in that case
$\chi_D$ should not be called a Dirichlet character modulo $D$ alone.
The finite Euler product $\mathcal L(u,\chi_D)$ is nevertheless exactly the
numerator of the zeta function of $y^2=D(T)$ (the unique place at infinity is
ramified because $d(D)$ is odd); see \cite{Ro02}*{Chapter~9} for the
Artin--Hecke interpretation.  It is therefore a polynomial of degree
$2g$ satisfying
\begin{equation*}
\mathcal L(u,\chi_D)=(qu^2)^g\mathcal L\!\biggl(\frac{1}{qu},\chi_D\biggr).
\end{equation*}
Consequently, as in \cite{Fl17}*{Lemma~2.1}, we have 
\begin{equation}\label{eq:afe}
L\!\left(\tfrac12,\chi_D\right)
=\sum_{f\in\A^+_{\le g}}\frac{\chi_D(f)}{\sqrt{|f|}}+\sum_{f\in\A^+_{\le g-1}}\frac{\chi_D(f)}{\sqrt{|f|}}.
\end{equation}

For $f\in\A^+$, define the quadratic character $\psi_f:\A\to\{0,\pm1\}$ by
$\psi_f(H)=\bigl(\frac Hf\bigr)$, so that $\chi_D(f)=\psi_f(D)$ for every nonconstant $D\in\Hh$.
The square-free sieve of \cite{Fl17}*{Lemma~2.2} gives
\begin{align}\label{eq:character-average}
\sum_{D\in\Hh_{2g+1}}\chi_D(f)
=\sum_{\substack{C\in\A^+_{\le g}\\C\mid f^\infty}}\sum_{H\in\A^+_{2g+1-2d(C)}}\psi_f(H)
-q\sum_{\substack{C\in\A^+_{\le g-1}\\C\mid f^\infty}}\sum_{H\in\A^+_{2g-1-2d(C)}}\psi_f(H),
\end{align}
where $C\mid f^\infty$ means that every prime polynomial divisor of $C$ divides $f$. 

\subsection{Phase-normalized Poisson summation}\label{subsec:phase-poisson}

For an arbitrary odd prime power $q$ the quadratic Gauss sums and the Poisson formula
require modification, as noted in \cite{Fl17}*{Section~9}.  We handle this by a phase
normalization built into the definition of the Gauss sum, so that a single Gauss sum
occurs throughout and the phase-free local evaluations and multiplicativity used in
Florea's argument are restored for every odd prime power $q$.

Let $p=\operatorname{char}\mathbb F_q$ and define a nontrivial additive character $\varphi:\mathbb F_q\to\mathbb C^\times$ by
$$
\varphi(\alpha)=\exp\!\left(\frac{2\pi i}{p}\operatorname{Tr}_{\mathbb F_q/\mathbb F_p}(\alpha)\right).
$$
For $a\in\mathbb F_q((T^{-1}))$ write $[T^{-1}]a$ for the coefficient of $T^{-1}$ in the
expansion of $a$, and put $e(a)=\varphi\bigl([T^{-1}]a\bigr)$; thus $e(a)=1$ for $a\in\A$.
Let $\chi_2$ be the quadratic character of $\mathbb F_q^\times$ and set
\begin{equation*}
\epsilon_q=q^{-1/2}\sum_{\alpha\in\mathbb F_q^\times}\chi_2(\alpha)\varphi(\alpha).
\end{equation*}
Then $\epsilon_q^2=\chi_2(-1)=\eta_q$ and $\epsilon_q^4=1$;
see \cite{IR90}*{Chapter~10, Section~3}.

For $f\in\A^{+}$, let $\nu(f)\in\{0,1\}$ be the parity of $d(f)$.
For $f\in\A^+$ and $V\in\A$, define the phase-normalized quadratic Gauss sum
\begin{equation}\label{eq:Gc-definition-odd}
\Gc(V,\psi_f):=\overline{\epsilon_q}^{\,\nu(f)}
\sum_{u\bmod f}\psi_f(u)\,e\!\left(\frac{uV}{f}\right).
\end{equation}
Here $u$ runs over a complete residue system modulo $f$; the summand depends only on
$u\bmod f$, because $e(a)=1$ for $a\in\A$.
When $d(f)$ is even the normalizing factor is $1$, so that $\Gc(V,\psi_f)$ is then the
quadratic Gauss sum used in \cite{Fl17}; when $d(f)$ is odd the two differ by the
fourth root of unity $\overline{\epsilon_q}$.
Only $\Gc$ occurs in the remainder of the paper.

\begin{lemma}\label{lem:phase-free}
\begin{enumerate}
\item 
For coprime $f_1,f_2\in\A^+$, one has
\begin{equation}\label{eq:Gc-multiplicativity-odd}
\Gc(V,\psi_{f_1f_2})=\Gc(V,\psi_{f_1})\Gc(V,\psi_{f_2}).
\end{equation}

\item 
Let $P\in\Pp$, $i\ge1$, and let $V\ne0$. 
Write $V=P^kV_1$ with $P\nmid V_1$. Then
\begin{equation}\label{eq:local-gauss-table}
\Gc(V,\psi_{P^i})=
\begin{cases}
0,&i\le k,\ i\ \mathrm{odd},\\
\phi(P^i),&i\le k,\ i\ \mathrm{even},\\
-|P|^{i-1},&i=k+1,\ i\ \mathrm{even},\\
\left(\frac{V_1}{P}\right)|P|^{i-1}\sqrt{|P|}, &i=k+1,\ i\ \mathrm{odd},\\
0,&i\ge k+2.
\end{cases}
\end{equation}

\item 
For every $f\in\A^+$, one has
\begin{equation}\label{eq:Gc-zero-odd}
\Gc(0,\psi_f)=\phi(f)\mathbf 1_{f=\square}.
\end{equation}
\end{enumerate}
\end{lemma}

\begin{proof}
We begin with (1).  Let $f_1,f_2\in\A^{+}$ be coprime and write $u=u_1f_2+u_2f_1$, where
$u_j$ runs over a complete residue system modulo $f_j$; then $u$ runs over one modulo
$f_1f_2$, and $uV/(f_1f_2)=u_1V/f_1+u_2V/f_2$.  Since $u\equiv u_1f_2\bmod f_1$ and
$u\equiv u_2f_1\bmod f_2$, quadratic reciprocity gives
$$
\psi_{f_1f_2}(u)=\left(\frac{f_2}{f_1}\right)\left(\frac{f_1}{f_2}\right)
\psi_{f_1}(u_1)\psi_{f_2}(u_2)=\eta_q^{\,d(f_1)d(f_2)}\psi_{f_1}(u_1)\psi_{f_2}(u_2).
$$
Hence, for the sums occurring in \eqref{eq:Gc-definition-odd},
$$
\sum_{u\bmod f_1f_2}\psi_{f_1f_2}(u)\,e\!\left(\frac{uV}{f_1f_2}\right)
=\eta_q^{\,d(f_1)d(f_2)}\prod_{j=1}^{2}
\sum_{u\bmod f_j}\psi_{f_j}(u)\,e\!\left(\frac{uV}{f_j}\right).
$$
Multiplying both sides by $\overline{\epsilon_q}^{\,\nu(f_1f_2)}$ and using
$\eta_q=\epsilon_q^2$, we obtain
$$
\Gc(V,\psi_{f_1f_2})
=\epsilon_q^{\,\nu(f_1)+\nu(f_2)-\nu(f_1f_2)+2d(f_1)d(f_2)}\,
\Gc(V,\psi_{f_1})\,\Gc(V,\psi_{f_2}).
$$
Since $\nu(f_1f_2)=\nu(f_1)+\nu(f_2)-2\nu(f_1)\nu(f_2)$, the exponent equals
$2\bigl(\nu(f_1)\nu(f_2)+d(f_1)d(f_2)\bigr)$, and
$\nu(f_1)\nu(f_2)\equiv d(f_1)d(f_2)\bmod2$; hence the exponent is divisible by $4$.
As $\epsilon_q^4=1$, this proves \eqref{eq:Gc-multiplicativity-odd}.
In other words, the parity function $\nu$ satisfies
$\nu(f_1f_2)\equiv\nu(f_1)+\nu(f_2)+2d(f_1)d(f_2)\bmod4$,
which is exactly the relation needed to absorb the reciprocity phase.

We next evaluate $\Gc$ at a prime; this is the only point in the paper at which the
constant field enters.
Let $P\in\Pp$, put $d=d(P)$, and let $\theta$ be a zero of $P$.  For $u\in\A$ with
$d(u)<d$, expanding $u/P$ in partial fractions over $\mathbb F_{q^d}$ and reading off
the coefficient of $T^{-1}$ gives the residue--trace identity
\begin{equation*}
 e(u/P)=\varphi\!\left(
 \operatorname{Tr}_{\mathbb F_{q^d}/\mathbb F_q}
 \frac{u(\theta)}{P'(\theta)}\right).
\end{equation*}
Hence, if $\chi_{2,d}=\chi_2\circ N_{\mathbb F_{q^d}/\mathbb F_q}$ and $\tau_{q^d}$
denotes the quadratic Gauss sum on $\mathbb F_{q^d}$ with additive character
$\varphi\circ\operatorname{Tr}_{\mathbb F_{q^d}/\mathbb F_q}$, then
\eqref{eq:Gc-definition-odd} gives
\begin{equation}\label{eq:Gc-prime-trace}
 \Gc(1,\psi_P)=\overline{\epsilon_q}^{\,\nu(P)}\,\chi_{2,d}\bigl(P'(\theta)\bigr)\,\tau_{q^d}.
\end{equation}
The Hasse--Davenport lifting relation
\cite{IR90}*{Chapter~11, Section~4, pp.~163--165} gives
\begin{equation}\label{eq:HD-quadratic}
 \tau_{q^d}=(-1)^{d-1}\epsilon_q^d q^{d/2}.
\end{equation}
On the other hand, $\operatorname{disc}(P)=(-1)^{d(d-1)/2}N_{\mathbb F_{q^d}/\mathbb F_q}(P'(\theta))$, so the
Pellet--Stickelberger parity formula \cite{Sw62}*{Corollary~1}, applied to the
irreducible polynomial $P$, gives
\begin{equation}\label{eq:pellet-derivative}
 \chi_{2,d}\bigl(P'(\theta)\bigr)=(-1)^{d-1}\eta_q^{d(d-1)/2}.
\end{equation}
Substituting \eqref{eq:HD-quadratic} and \eqref{eq:pellet-derivative} into
\eqref{eq:Gc-prime-trace}, the two factors $(-1)^{d-1}$ cancel, and
$\eta_q=\epsilon_q^2$ turns $\eta_q^{d(d-1)/2}\epsilon_q^{d}$ into $\epsilon_q^{d^2}$.
Therefore
\begin{equation}\label{eq:Gc-prime-phase}
 \Gc(1,\psi_P)=\epsilon_q^{\,d^2-\nu(P)}\sqrt{|P|}=\sqrt{|P|},
\end{equation}
because $d^2\equiv\nu(P)\bmod4$ for every integer $d$, both sides being $0$ for $d$
even and $1$ for $d$ odd, and because $\epsilon_q^4=1$.
Thus the normalization in \eqref{eq:Gc-definition-odd} cancels the constant-field
phase exactly.  Compare \cite{Fl17}*{Section~9}, where the corresponding phase is
computed under the assumption that $q\equiv1\bmod4$.

We can now prove (2).  The reduction of the sum in \eqref{eq:Gc-definition-odd} at a
prime power to the case $i=1$ is the polynomial Ramanujan-sum computation of
\cite{Fl17}*{Lemma~3.2}, which uses no property of $q$ beyond oddness.  
Writing $V=P^kV_1$ with $P\nmid V_1$, that computation shows that the unnormalized sum
in \eqref{eq:Gc-definition-odd} with $f=P^i$ vanishes when $i\le k$ with $i$ odd and
when $i\ge k+2$, that it equals $\phi(P^i)$ when $i\le k$ with $i$ even, that it equals
$-|P|^{i-1}$ when $i=k+1$ is even, and that it equals $\left(\frac{V_1}{P}\right)|P|^{i-1}$
times the corresponding sum with $V=1$ and $f=P$ when $i=k+1$ is odd. 
In the first four alternatives the value either vanishes or has $i$ even, so
that $\nu(P^i)=0$ and the normalizing factor in \eqref{eq:Gc-definition-odd} is
trivial.  In the remaining alternative $i$ is odd, so that $\nu(P^i)=\nu(P)$ and the
normalizing factor $\overline{\epsilon_q}^{\,\nu(P^i)}$ equals the one attached to
$\Gc(1,\psi_P)$; hence \eqref{eq:Gc-prime-phase} gives
$$
\Gc(V,\psi_{P^i})=\left(\frac{V_1}{P}\right)|P|^{i-1}\Gc(1,\psi_P)
=\left(\frac{V_1}{P}\right)|P|^{i-1}\sqrt{|P|}.
$$
This proves \eqref{eq:local-gauss-table}.

Finally we prove (3).  Since $e(0)=1$, the sum in \eqref{eq:Gc-definition-odd} with
$V=0$ is $\sum_{u\bmod f}\psi_f(u)$.  The character $\psi_f$ is principal modulo $f$
precisely when $f$ is a square, in which case $d(f)$ is even, the normalizing factor
is trivial, and the sum equals $\phi(f)$; otherwise the sum vanishes.  This proves
\eqref{eq:Gc-zero-odd}.
\end{proof}

Lemma~\ref{lem:phase-free} allows the Poisson summation formula to be stated, for every
odd prime power $q$, with the single Gauss sum $\Gc$ throughout.

\begin{proposition}[Poisson summation]\label{prop:poisson}
Let $f\in\A^+$, and $m\ge0$.  
If $d(f)$ is even, then
\begin{align}\label{eq:poisson-even}
\sum_{H\in\A_m^+}\psi_f(H)
=\frac{q^m}{|f|}\Biggl(\Gc(0,\psi_f)+(q-1)\sum_{V\in\A^+_{\le d(f)-m-2}}\Gc(V,\psi_f) 
-\sum_{V\in\A^+_{d(f)-m-1}}\Gc(V,\psi_f)\Biggr).
\end{align}
If $d(f)$ is odd, then
\begin{equation}\label{eq:poisson-odd}
\sum_{H\in\A_m^+}\psi_f(H)=\frac{q^m\sqrt q}{|f|}\sum_{V\in\A^+_{d(f)-m-1}}\Gc(V,\psi_f).
\end{equation}
\end{proposition}

\begin{proof}
The argument is that of \cite{Fl17}*{Proposition~3.1}, carried out with $\Gc$ in place of
the Gauss sum used there.
Suppose first that $m\ge d(f)$.  Every residue class modulo $f$ then has exactly
$q^{m-d(f)}$ monic representatives of degree $m$, so that
$$
\sum_{H\in\A_m^+}\psi_f(H)=q^{m-d(f)}\sum_{u\bmod f}\psi_f(u),
$$ 
and \eqref{eq:Gc-zero-odd} gives \eqref{eq:poisson-even}; when $d(f)$ is odd, $f$ is not a
square and both sides of \eqref{eq:poisson-odd} vanish.

Suppose that $0\le m\le d(f)-1$.  
Additive-character orthogonality gives
\begin{equation}\label{eq:poisson-proof-start}
\sum_{H\in\A_m^+}\psi_f(H)
=\frac{\psi_f(-1)\,\epsilon_q^{\,\nu(f)}}{|f|}\sum_{u\bmod f}S_m(u)\,\Gc(u,\psi_f),  \quad
\text{ where }~~ S_m(u):=\sum_{H\in\A_m^+}e\!\left(\frac{u H}{f}\right),
\end{equation}
the factor $\epsilon_q^{\,\nu(f)}$ arising because the Fourier coefficient of $\psi_f$ is
the unnormalized sum in \eqref{eq:Gc-definition-odd}.
For $\alpha\in\mathbb F_q^\times$, we have $\psi_f(\alpha)=\chi_2(\alpha)^{d(f)}$,
$\psi_f(-1)=\eta_q^{d(f)}$ and $\Gc(\alpha V,\psi_f)=\psi_f(\alpha)\Gc(V,\psi_f)$,
the last relation being unaffected by the constant normalizing factor.
Writing $H=T^m+\sum_{j<m}\alpha_jT^j$ and summing over the coefficients
$\alpha_j\in\mathbb F_q$, additive-character orthogonality gives
$$
S_m(u)=
\begin{cases}
q^m\,\varphi\!\left([T^{-1}]\dfrac{uT^m}{f}\right), & d(u)\le d(f)-m-1,\\[2mm]
0,&\text{otherwise};
\end{cases}
$$
in particular $S_m(0)=q^m$.
The term $u=0$ contributes $q^m\Gc(0,\psi_f)$.  Every nonzero residue with
$S_m(u)\ne0$ is uniquely of the form $u=\alpha V$ with $\alpha\in\mathbb F_q^\times$ and
$V\in\A^+$ of degree $d(V)\le d(f)-m-1$; the sum over these $u$ therefore becomes a
double sum over $V$ and $\alpha$, and $\Gc(\alpha V,\psi_f)=\psi_f(\alpha)\Gc(V,\psi_f)$
separates the two variables.  If $d(V)\le d(f)-m-2$, then
$[T^{-1}]\frac{\alpha VT^m}{f}=0$ and $S_m(\alpha V)=q^m$; if $d(V)=d(f)-m-1$, then the
monicity of $f$ and $V$ gives $[T^{-1}]\frac{\alpha VT^m}{f}=\alpha$ and
$S_m(\alpha V)=q^m\varphi(\alpha)$.  Hence
\begin{equation}\label{eq:poisson-split}
\sum_{u\bmod f}S_m(u)\Gc(u,\psi_f)
=q^m\left(\Gc(0,\psi_f)
+\Sigma_0\sum_{V\in\A^+_{\le d(f)-m-2}}\Gc(V,\psi_f)
+\Sigma_1\sum_{V\in\A^+_{d(f)-m-1}}\Gc(V,\psi_f)\right),
\end{equation}
where
$$
\Sigma_0:=\sum_{\alpha\in\mathbb F_q^\times}\psi_f(\alpha),\qquad
\Sigma_1:=\sum_{\alpha\in\mathbb F_q^\times}\psi_f(\alpha)\varphi(\alpha).
$$

If $d(f)$ is even, then $\psi_f(-1)=1$, $\psi_f(\alpha)=1$ for every
$\alpha\in\mathbb F_q^\times$, and $\nu(f)=0$, so that the prefactor in
\eqref{eq:poisson-proof-start} is $1$.  Moreover $\Sigma_0=q-1$ and
$\Sigma_1=\sum_{\alpha\in\mathbb F_q^\times}\varphi(\alpha)=-1$, and
\eqref{eq:poisson-split} gives \eqref{eq:poisson-even}.

If $d(f)$ is odd, then $f$ is not a square, so that $\Gc(0,\psi_f)=0$ by
\eqref{eq:Gc-zero-odd} and the first term of \eqref{eq:poisson-split} is absent.  Here
$\psi_f(\alpha)=\chi_2(\alpha)$, so that
$$
\Sigma_0=\sum_{\alpha\in\mathbb F_q^\times}\chi_2(\alpha)=0,\qquad
\Sigma_1=\sum_{\alpha\in\mathbb F_q^\times}\chi_2(\alpha)\varphi(\alpha)=\epsilon_q\sqrt q,
$$
and only the sum over $V$ of degree exactly $d(f)-m-1$ survives.  Moreover
$\psi_f(-1)=\eta_q=\epsilon_q^2$ and $\nu(f)=1$, so that the prefactor in
\eqref{eq:poisson-proof-start} equals $\epsilon_q^3/|f|$.  Multiplying it by the factor
$\epsilon_q\sqrt q$ carried by $\Sigma_1$ gives $\epsilon_q^4\sqrt q/|f|=\sqrt q/|f|$,
since $\epsilon_q^4=1$; together with the factor $q^m$ in \eqref{eq:poisson-split} this proves \eqref{eq:poisson-odd}.
\end{proof}

\subsection{Decomposition of the moment}\label{subsec:moment-decomposition}

For $N\in\{g,g-1\}$, the two ranges occurring in \eqref{eq:afe}, define
\begin{equation}\label{eq:SN-odd-definition}
\mathcal S_N:=\sum_{f\in\A^+_{\le N}}\frac{1}{\sqrt{|f|}}\sum_{D\in\Hh_{2g+1}}\chi_D(f).
\end{equation}
Substituting \eqref{eq:character-average} into \eqref{eq:SN-odd-definition},
and separating the single endpoint $d(C)=g$ in the first sum, we have 
\begin{align}\label{eq:SN-after-sieve}
\mathcal S_N
=\sum_{f\in\A^+_{\le N}}\frac1{\sqrt{|f|}}\sum_{\substack{C\in\A^+_{\le g-1}\\C\mid f^\infty}}
\Bigg(\sum_{H\in\A^+_{2g+1-2d(C)}}\psi_f(H)-q\sum_{H\in\A^+_{2g-1-2d(C)}}\psi_f(H)\Bigg)
+O_{q,\e}\!\left(q^{N/2+\e g}\right).
\end{align}
Indeed, when $d(C)=g$ the inner sum is $\sum_{H\in\A_1^+}\psi_f(H)$, of absolute value at
most $q$; together with the divisor bound
$\#\{C\in\A_g^+:C\mid f^\infty\}\ll_\e q^{\e(g+d(f))}$ and
$\sum_{f\in\A^+_{\le N}}|f|^{-1/2}\ll q^{N/2}$, this bounds the omitted endpoint by
$q^{N/2+\e g+\e N}$, which is $O_{q,\e}(q^{N/2+\e g})$ after renaming $\e$.

Split the main term of \eqref{eq:SN-after-sieve} according to the parity of $d(f)$,
writing $\mathcal S_{N,\rm e}$ for the part with $d(f)$ even and $\mathcal S_{N,\rm o}$
for the part with $d(f)$ odd, so that
$\mathcal S_N=\mathcal S_{N,\rm e}+\mathcal S_{N,\rm o}+O_{q,\e}\bigl(q^{N/2+\e g}\bigr)$
with the same error term as in \eqref{eq:SN-after-sieve}. 
We use the Poisson summation formula \eqref{eq:poisson-even} for the sum over $H$ to write $\mathcal S_{N,\rm e}$ as
\begin{align} \label{eq:SN-even-sector} 
\mathcal S_{N,\mathrm e} 
&= q^{2g} \sum_{\substack{f\in\A^+_{\le N}\\ d(f)\ {\rm even}}} \frac{1}{|f|^{3/2}} 
\sum_{\substack{C\in\A^+_{\le g-1}\\ C\mid f^\infty}} \frac{1}{|C|^2} 
\Bigg( (q-1)\Gc(0,\psi_f) +q(q-1) \sum_{V\in\A^+_{\le d(f)-2g+2d(C)-3}} \Gc(V,\psi_f) \notag\\ 
&\hspace{-0.5em} -q \sum_{V\in\A^+_{d(f)-2g+2d(C)-2}} \Gc(V,\psi_f) -(q-1) \sum_{V\in\A^+_{\le d(f)-2g+2d(C)-1}} \Gc(V,\psi_f) 
+ \sum_{V\in\A^+_{d(f)-2g+2d(C)}} \Gc(V,\psi_f) \Bigg). 
\end{align}
We decompose $\mathcal S_{N,\mathrm e}$ according to the nature of the dual variable $V$ in \eqref{eq:SN-even-sector} 
as $\mathcal S_{N,\mathrm e}=\mathcal M_N+\mathcal S_N(V=\square)+\widetilde{\mathcal S}_N(V\ne\square)$, where 
$\mathcal M_N$ denotes the contribution of the zero frequency $V=0$,
$\mathcal S_N(V=\square)$ denotes the contribution of the nonzero square frequencies,
and $\widetilde{\mathcal S}_N(V\ne\square)$ denotes the contribution of the nonsquare frequencies. 

We use the Poisson summation formula \eqref{eq:poisson-odd} for the sum over $H$ to write $\mathcal S_{N,\rm o}$ as 
\begin{align*}
\mathcal S_{N,\rm o}=q^{2g+3/2}
\sum_{\substack{f\in\A^+_{\le N}\\d(f)\ \mathrm{odd}}}\frac1{|f|^{3/2}}
\sum_{\substack{C\in\A^+_{\le g-1}\\C\mid f^\infty}}\frac1{|C|^2}
\Bigg(\sum_{V\in\A^+_{d(f)-2g+2d(C)-2}}\Gc(V,\psi_f)
-\frac1q \sum_{V\in\A^+_{d(f)-2g+2d(C)}}\Gc(V,\psi_f)\Bigg).
\end{align*}
In the preceding equation, when $d(f)$ is odd, $d(V)$ is also odd, so $V$ cannot be a square. 
Define $\mathcal S_N(V\ne\square) = \mathcal S_{N,\rm o}+\widetilde{\mathcal S}_N(V\ne\square)$.

By \eqref{eq:afe} we have $M_1(g)=\mathcal S_g+\mathcal S_{g-1}$; combining this with the
two decompositions above gives, for every $\e>0$,
\begin{align}\label{eq:moment-decomposition-short}
M_1(g)=\sum_{N\in\{g,g-1\}}\bigl(\mathcal M_N+\mathcal S_N(V=\square)+\mathcal S_N(V\ne\square)\bigr)
+O_{q,\e}\!\left(q^{g(1+\e)/2}\right).
\end{align}

\section{Main term, dual contributions, and proof of the main theorem}
\label{sec:principal-nonsquare}

The aim of this section is to prove Theorem~\ref{thm:main}, granted the square-dual
evaluation stated below as Proposition~\ref{prop:square-dual-evaluation} and proved in
\S\ref{sec:periodic-term}.  Starting from
\eqref{eq:moment-decomposition-short}, we evaluate the zero-frequency contribution,
record the square-dual evaluation, and bound the nonsquare-dual contribution.  The
zero-frequency evaluation is deliberately left incomplete: it produces a contour
integral $\mathcal R_0(g;\rho)$ which is retained rather than estimated, because the
square-dual contribution produces $-\mathcal R_0(g;\rho)$ and the two cancel exactly in
the final assembly.

\subsection{Zero-frequency contribution}

By \eqref{eq:Gc-zero-odd}, the zero frequency occurs only when $f$ is a square, so that
$d(f)$ is even, $\nu(f)=0$ and $\Gc(0,\psi_f)=\phi(f)$.  The phase normalization of
\S\ref{sec:standard-reductions} therefore plays no role here, and the main-term
argument of \cite{Fl17}*{Section~5} applies verbatim.

Put
\begin{equation*}
\mathcal C(u):=\prod_{P\in\Pp}\left(1-\frac{u^{d(P)}}{|P|(|P|+1)}\right),
\end{equation*}
a product that converges absolutely and locally uniformly for $|u|<q$ and therefore
defines a holomorphic nonvanishing function there.  Set
\begin{equation}\label{eq:principal-term}
\mathcal Q_g:=\frac{\mathscr P(1)}{2\zA(2)}q^{2g+1}\left(2g+2+\frac4{\log q}\frac{\mathscr P'}{\mathscr P}(1)\right),
\end{equation}
and, for $1<\rho<q$,
\begin{equation}\label{eq:R0}
\mathcal R_0(g;\rho):=\frac{q^{2g+1}}{\zA(2)}\sum_{N\in\{g,g-1\}}\frac1{2\pi i}
\oint_{|u|=\rho}\frac{\mathcal C(u)}{(1-u)^2u^{\lfloor N/2\rfloor+1}}\,du.
\end{equation}

\begin{proposition}\label{prop:zero-frequency}
For every $\e>0$ and every $\rho$ with $1<\rho<q$, one has
\begin{equation*}
\mathcal M_g+\mathcal M_{g-1}=\mathcal Q_g+\mathcal R_0(g;\rho)+O_{q,\e}(q^{\e g}).
\end{equation*}
\end{proposition}

\begin{proof}
Writing $f=l^2$, completing the divisor sum by the Rankin estimate of
\cite{Fl17}*{Section~5}, and rearranging the resulting Euler product reduce
$\mathcal M_N$ to a contour integral in the single variable $u$: for $N\in\{g,g-1\}$ and
$0<r<1$,
\begin{equation}\label{eq:M-N-one-variable}
\mathcal M_N=\frac{q^{2g+1}}{\zA(2)}\frac1{2\pi i}
\oint_{|u|=r}\frac{\mathcal C(u)}{(1-u)^2u^{\lfloor N/2\rfloor+1}}\,du+O_{q,\e}(q^{\e g}),
\end{equation}
where $u$ is scaled so that the principal pole lies at $u=1$ rather than at $u=q^{-1}$.
Everything up to \eqref{eq:M-N-one-variable} is taken from \cite{Fl17}*{Section~5}; the
remainder of the proof, namely shifting the contour across the double pole at $u=1$ and
summing over $N$, is carried out here.

Move the contour in \eqref{eq:M-N-one-variable} from $|u|=r$ to $|u|=\rho$.  Since
$\mathcal C$ is holomorphic and nonvanishing on $|u|<q$ and the only pole crossed is the
double pole at $u=1$, 
$$
-\Res_{u=1}\frac{\mathcal C(u)}{(1-u)^2u^{\lfloor N/2\rfloor+1}}=(\lfloor N/2\rfloor+1)\mathcal C(1)-\mathcal C'(1).
$$
Hence
\begin{equation}\label{eq:M-N-residue-decomposition}
\mathcal M_N
=\frac{q^{2g+1}}{\zA(2)}\left(\bigl(\lfloor N/2\rfloor+1\bigr)\mathcal C(1)-\mathcal C'(1)
+\frac1{2\pi i}\oint_{|u|=\rho}\frac{\mathcal C(u)}{(1-u)^2u^{\lfloor N/2\rfloor+1}}\,du\right)+O_{q,\e}(q^{\e g}).
\end{equation}

Sum \eqref{eq:M-N-residue-decomposition} over $N\in\{g,g-1\}$.  Since
$\lfloor g/2\rfloor+\lfloor(g-1)/2\rfloor=g-1$, the two residue terms contribute
$\frac{q^{2g+1}}{\zA(2)}\bigl((g+1)\mathcal C(1)-2\mathcal C'(1)\bigr)$.  Comparing the
Euler products defining $\mathscr P$ and $\mathcal C$ gives
$\mathscr P(s)=\mathcal C(q^{1-s})$, whence $\mathcal C(1)=\mathscr P(1)$ 
and $-\frac{\mathcal C'(1)}{\mathcal C(1)}=\frac1{\log q}\frac{\mathscr P'}{\mathscr P}(1)$,
and therefore this contribution equals $\mathcal Q_g$.  The two displaced integrals sum
to $\mathcal R_0(g;\rho)$ by \eqref{eq:R0}, and the two error terms are
$O_{q,\e}(q^{\e g})$; the proposition follows.
\end{proof}

We retain the displaced contour $\mathcal R_0(g;\rho)$ explicitly because it will cancel
with the principal square-dual contribution in \S\ref{subsec:secondary-reduction}.

\subsection{Square-dual contribution}\label{subsec:square-dual}

The nonzero square-dual contribution is the source of the periodic secondary term in Theorem~\ref{thm:main}.  
In the odd-degree sector of \S\ref{subsec:moment-decomposition} the dual ranges
satisfy $d(V)\equiv d(f)\bmod2$, so every dual variable arising there is a nonsquare.
Hence nonzero square frequencies occur only in the even-degree sector \eqref{eq:SN-even-sector}.  
For such moduli, $\nu(f)=0$ in \eqref{eq:Gc-definition-odd}, so that $\Gc(V,\psi_f)$
is the unnormalized quadratic Gauss sum used in \cite{Fl17}*{Section~6}.
Thus the phase normalization required for an arbitrary odd prime power does not affect the square-dual calculation.

The evaluation is the counterpart of \cite{Fl17}*{Section~6}.  
The cumulative rearrangement of the square ranges, the associated bivariate Euler product, and the deformation of 
the principal contour are the same as in Florea's argument after the change of variables and notation used here.  
The principal square-dual branch cancels the retained zero-frequency contour $\mathcal R_0(g;\rho)$.  
The point at which the present analysis differs is the secondary-pole bookkeeping:
the secondary circle contains three double poles of the same modulus, and all three must
be retained.

With $\alpha_m$ and $\beta_m$ as in \eqref{eq:alpha-periodic} and
\eqref{eq:beta-periodic}, set
\begin{equation}\label{eq:Sodd-def}
\Sodd(g):=q^{2g/3}\bigl(\alpha_{g\bmod3}\,g+\beta_{g\bmod3}\bigr).
\end{equation}
The following proposition records the square-dual evaluation used in the proof of
Theorem~\ref{thm:main}.  Its proof, including the Euler-product reduction and
the residue calculation at all three secondary points, is given in \S\ref{sec:periodic-term}.

\begin{proposition} \label{prop:square-dual-evaluation}
For every $\e>0$ and every $\rho$ with $1<\rho<q$, one has
\begin{equation*}
\mathcal R_0(g;\rho)+\mathcal S_g(V=\square)+\mathcal S_{g-1}(V=\square)
=\Sodd(g)+O_{q,\e}\!\left(q^{g(1+\e)/2}\right).
\end{equation*}
\end{proposition}

\subsection{Nonsquare-dual contribution}

We now estimate the nonsquare-dual contributions $\mathcal S_N(V\ne\square)$, for $N\in\{g,g-1\}$, 
defined in \S\ref{subsec:moment-decomposition}. 
This is the nonsquare-dual calculation of \cite{Fl17}*{Section~7}.  
Florea observes in \cite{Fl17}*{Section~9} that, 
when the primality assumption on $q$ is removed while $q\equiv1\bmod 4$ is retained, 
the required modifications occur in the Poisson formula and in the square-dual calculation, 
but not in the nonsquare estimate.

For the nonsquare estimate over an arbitrary odd prime power $q$,
the only additional issue is the quadratic Gauss phase.  
In the present normalization, this issue has already been resolved in \S\ref{sec:standard-reductions}.  
Indeed, Lemma~\ref{lem:phase-free} shows that the phase-normalized Gauss sum $\Gc(V,\psi_f)$ 
satisfies the same prime-power evaluations and multiplicativity relations 
as Florea's quadratic Gauss sum used in \cite{Fl17}*{Section~7}. 
Proposition~\ref{prop:poisson} also gives the Poisson formula in the same phase-free form.

For completeness, the precise transfer to the fixed-nonsquare Euler products
is recorded next.

\begin{lemma}[Nonsquare local transfer]\label{lem:nonsquare-transfer}
Let $V$ be nonsquare and let $V_0\ne1$ be its square-free part, so that $V/V_0$ is a
square.  Put
\begin{equation*}
 \Bns(V;w,u):=
 \sum_{f\in\A^+}w^{d(f)}
 \frac{\Gc(V,\psi_f)}{\sqrt{|f|}}
 \prod_{P\mid f}(1-u^{d(P)})^{-1}.
\end{equation*}
Then:
\begin{enumerate}
\item by \eqref{eq:Gc-multiplicativity-odd} the summand of $\Bns(V;w,u)$ is
multiplicative in $f$, so that $\Bns(V;w,u)$ has an Euler product whose local factors are
obtained from the table \eqref{eq:local-gauss-table};
\item for $P\nmid V$ the local factor is
$1+\left(\frac{V_0}{P}\right)w^{d(P)}(1-u^{d(P)})^{-1}$,
and the factor remaining after the first extraction, that of
$\mathcal L(w,\chi_{V_0})$, is
\begin{equation}\label{eq:nonsquare-unramified-factor}
 1+\frac{\left(\frac{V_0}{P}\right)(uw)^{d(P)}}{1-u^{d(P)}}
   -\frac{w^{2d(P)}}{1-u^{d(P)}},
\end{equation}
whereas for $P\mid V$ it is
\begin{equation}\label{eq:nonsquare-ramified-factor}
 1+\frac1{1-u^{d(P)}}
 \sum_{j\ge1}\frac{\Gc(V,\psi_{P^j})w^{jd(P)}}{|P|^{j/2}};
\end{equation}
\item if $k$ is minimal with $|wu^k|<q^{-1}$, then repeated extraction from
\eqref{eq:nonsquare-unramified-factor}, which produces the successive factors
$\mathcal L(wu,\chi_{V_0}),\mathcal L(wu^2,\chi_{V_0}),\dots$, gives in total
\begin{equation}\label{eq:nonsquare-L-extraction}
 \Bns(V;w,u)=
 \left(\prod_{j=0}^{k-1}\mathcal L(wu^j,\chi_{V_0})\right)\Cns(V;w,u),
\end{equation}
up to the finitely many ramified Euler factors at $P\mid V$; with $k$ as above, the
product defining $\Cns$ is absolutely convergent on the contours
$|w|=q^{-1/2-\e}$ and $|u|=q^{-\e}$ used in \cite{Fl17}*{Section~7}.
\end{enumerate}
\end{lemma}

\begin{proof}
Parts (1) and (2) follow prime by prime from
\eqref{eq:local-gauss-table}; crucially, the surviving odd exponent has no
constant-field phase.  
Multiplying \eqref{eq:nonsquare-unramified-factor} over $P\nmid V$ extracts
$\mathcal L(wu,\chi_{V_0})$ and leaves a local factor of the same shape with $u^2w$ in
place of $uw$, together with terms that are $O(|w|^{2d(P)})$ and hence summable over $P$,
since $|w|^2<q^{-1}$.  Iterating until $|wu^k|<q^{-1}$ yields
\eqref{eq:nonsquare-L-extraction}.  The factors at
$P\mid V$ are finite sums by the last line of
\eqref{eq:local-gauss-table}; more precisely, the factors
\eqref{eq:nonsquare-ramified-factor} terminate after finitely many terms, and the
standard divisor bound absorbs their product into $|V|^\delta$, for any fixed $\delta>0$.    
Thus these are term-for-term the local products used in the proof of \cite{Fl17}*{Lemma~7.1}.
\end{proof}

The remaining analytic input for this transfer is the following
off-critical-circle Lindel\"of bound.

\begin{lemma}[Off-critical-circle Lindel\"of bound]
Let $\eta,\theta>0$.  
Uniformly for square-free $W\in\A^+$ with $W\ne1$ and $|\xi|\le q^{-1/2-\eta}$, one has
\begin{equation}\label{eq:lindelof-fixed-q}
\mathcal L(\xi,\chi_W)\ll_{q,\eta,\theta}|W|^\theta.
\end{equation}
\end{lemma}

\begin{proof}
Let $X_W$ be the smooth projective hyperelliptic curve attached to $y^2=W(T)$, and let $h$ be its genus.  
By the Riemann hypothesis for curves, the reciprocal zeros of the numerator $P_W(t)$ of its zeta function all have modulus $\sqrt q$.
Thus
$$
P_W(t)=\prod_{j=1}^{2h}\left(1-\sqrt q\,e^{i\vartheta_j}t\right), \qquad (\vartheta_j\in\mathbb R).
$$
Up to the single Euler factor at infinity, $\mathcal L(t,\chi_W)$ equals $P_W(t)$; 
this factor and its reciprocal are bounded on $|t|\le q^{-1/2-\eta}$.  
Put
$$
\Sigma_m(W):=\sum_{j=1}^{2h}e^{im\vartheta_j}.
$$
The degree-two map $X_W\to\mathbb P^1$ and the trace formula give that $|\Sigma_m(W)|\le q^{m/2}+q^{-m/2}$ and $|\Sigma_m(W)|\le2h$.
Let $x=\sqrt q\,\xi$, so that $|x|\le q^{-\eta}<1$.  
Expanding the logarithm and splitting the sum at an integer $M\ge1$, we obtain
$$
\log|P_W(\xi)|\le \sum_{m\ge1}\frac{|\Sigma_m(W)|}{m}|x|^m\ll_{q,\eta} M+q^{\max\{1/2-\eta,0\}M} +\frac{hq^{-\eta M}}{M}.
$$
Taking $M=\max\{1,\lfloor\log_q(h+2)\rfloor\}$ makes the right-hand side $o(h)$.  
Since $d(W)=2h+O(1)$, the resulting bound for $\log|P_W(\xi)|$ is at most $\theta\log|W|$
once $d(W)$ is large enough in terms of $q$, $\eta$ and $\theta$; the finitely many
remaining $W$ are absorbed into the implied constant. 
This proves \eqref{eq:lindelof-fixed-q}.
\end{proof}

Estimate \eqref{eq:lindelof-fixed-q} is the off-critical-circle form of the
Littlewood--Lindel\"of estimate used in this argument; compare
\cite{AT14}*{Theorem~3.4}.

\begin{proposition}[Nonsquare estimate]\label{prop:nonsquare}
For every $\e>0$ and $N\in\{g,g-1\}$, one has 
\begin{equation*}
\mathcal S_N(V\ne\square)\ll_{q,\e}q^{N/2+\e g}.
\end{equation*}
Consequently,
\begin{equation*}
\mathcal S_g(V\ne\square)+\mathcal S_{g-1}(V\ne\square)\ll_{q,\e}q^{g(1+\e)/2}.
\end{equation*}
\end{proposition}

\begin{proof}
After replacing Florea's quadratic Gauss sum by $\Gc(V,\psi_f)$, the local Euler factors,
the fixed-nonsquare factorization and the extraction of the associated quadratic
$L$-functions are those of \cite{Fl17}*{Section~7}, by
Lemma~\ref{lem:nonsquare-transfer}, and the extracted $L$-functions are bounded by
\eqref{eq:lindelof-fixed-q}.  Moreover the odd-degree Poisson ranges consist entirely of
nonsquare dual variables, while in the even-degree ranges one simply restricts to
$V\ne\square$.  The argument proving \cite{Fl17}*{Lemma~7.1} therefore applies without
further modification and gives the first bound; the second follows by summing over
$N\in\{g,g-1\}$ and renaming $\e$.
\end{proof}

\subsection{Proof of the main theorem}

\begin{proof}[Proof of Theorem~\ref{thm:main}]
Fix $\e>0$ and $\rho$ with $1<\rho<q$.
By \eqref{eq:moment-decomposition-short},
\begin{align*}
M_1(g)&=\mathcal M_g+\mathcal M_{g-1}+\mathcal S_g(V=\square)+\mathcal S_{g-1}(V=\square) 
+\mathcal S_g(V\ne\square)+\mathcal S_{g-1}(V\ne\square)+O_{q,\e}\!\left(q^{g(1+\e)/2}\right).
\end{align*}
By Proposition~\ref{prop:zero-frequency},
$$
\mathcal M_g+\mathcal M_{g-1}=\mathcal Q_g+\mathcal R_0(g;\rho)+O_{q,\e}(q^{\e g}),
$$
while Proposition~\ref{prop:square-dual-evaluation}, rearranged, gives
$$
\mathcal S_g(V=\square)+\mathcal S_{g-1}(V=\square)
=\Sodd(g)-\mathcal R_0(g;\rho)+O_{q,\e}\!\left(q^{g(1+\e)/2}\right).
$$
Adding these two identities, the displaced contour $\mathcal R_0(g;\rho)$ cancels and
\begin{align*}
\mathcal M_g+\mathcal M_{g-1}+\mathcal S_g(V=\square)+\mathcal S_{g-1}(V=\square)
=\mathcal Q_g+\Sodd(g)+O_{q,\e}\!\left(q^{g(1+\e)/2}\right).
\end{align*}
Finally, Proposition~\ref{prop:nonsquare} yields
$$
\mathcal S_g(V\ne\square)+\mathcal S_{g-1}(V\ne\square)\ll_{q,\e}q^{g(1+\e)/2}.
$$
Substituting \eqref{eq:principal-term} for $\mathcal Q_g$ and \eqref{eq:Sodd-def} for
$\Sodd(g)$, and replacing $\e$ by a smaller positive parameter when necessary, proves
\eqref{eq:main-theorem}.  The reality of the coefficients and the exact minimal period of
the sequence $(\alpha_m,\beta_m)$ are established in
Proposition~\ref{prop:reality-period}.
\end{proof}

\section{The periodic secondary term}\label{sec:periodic-term}

We now return to the nonzero square-dual contribution
$\mathcal S_g(V=\square)+\mathcal S_{g-1}(V=\square)$ and prove
Proposition~\ref{prop:square-dual-evaluation}, together with the reality of the
coefficients and the exact minimal period of the sequence $(\alpha_m,\beta_m)$ asserted
in Theorem~\ref{thm:main} (Proposition~\ref{prop:reality-period}).
Subsection~\ref{subsec:secondary-reduction} follows \cite{Fl17}*{Section~6}, in the
normalization of \S\ref{subsec:phase-poisson}, and reduces the problem to two
one-variable contour integrals: a principal one, which cancels $\mathcal R_0(g;\rho)$,
and a secondary one.  The evaluation of the secondary integral, in
\S\ref{subsec:three-secondary-residues}, is where the present computation
departs from \cite{Fl17}: the secondary circle carries three double poles of equal
modulus, and all three contribute.

\subsection{Reduction to the secondary residues}\label{subsec:secondary-reduction}

Only even-degree $f$ occur in the square-dual contribution
(\S\ref{subsec:square-dual}); for such $f$ the normalizing factor in
\eqref{eq:Gc-definition-odd} is trivial, so that the local evaluations
\eqref{eq:local-gauss-table} are the ones used in \cite{Fl17}*{Section~6}.

We record the Euler-product reduction explicitly, both to fix our normalization and
because the same Euler products serve $N=g$ and $N=g-1$.  For $q^{-2}<|z|<q^{-1}$ and
$|w|<q^{-1}$, put
\begin{equation}\label{eq:B-bivariate}
 \mathcal B(z,w):=
 \sum_{f\in\A^+}w^{d(f)}
 \prod_{P\mid f}\left(1-\frac1{|P|^2z^{d(P)}}\right)^{-1}
 \sum_{Y\in\A^+}z^{d(Y)}\frac{\Gc(Y^2,\psi_f)}{\sqrt{|f|}}.
\end{equation}
It is essential that $\Gc$, and not the unnormalized sum of
\eqref{eq:Gc-definition-odd}, occur in \eqref{eq:B-bivariate}.  
Indeed, the outer sum runs over all $f\in\A^+$, whereas
by the proof of Lemma~\ref{lem:phase-free} the unnormalized sums satisfy
$$
\sum_{u\bmod f_1f_2}\psi_{f_1f_2}(u)\,e\!\left(\frac{uV}{f_1f_2}\right)
=\eta_q^{\,d(f_1)d(f_2)}\prod_{j=1}^2
\sum_{u\bmod f_j}\psi_{f_j}(u)\,e\!\left(\frac{uV}{f_j}\right),
$$
and hence fail to be multiplicative in $f$ whenever $\eta_q=-1$; the Euler-product
factorization of Lemma~\ref{lem:square-euler-products} below would then be false.  By
\eqref{eq:Gc-multiplicativity-odd} the normalized sum is multiplicative for every odd
prime power $q$.  In the application only even-degree $f$ occur, and there the
normalizing factor is trivial, so \eqref{eq:B-bivariate} coincides with the generating
function of \cite{Fl17}*{Section~6} on the relevant range.
If $d=d(P)$, define
\begin{align}
B_P(z,w):=1+\frac{1}{|P|^2z^d-1}\bigl\{
w^d-|P|^2(w^2z)^d-|P|^2(wz^2)^d+|P|^2(w^3z^2)^d +|P|(w^2z)^d-|P|(w^3z)^d\bigr\},
\label{eq:BP-explicit}
\end{align}
and
\begin{align}
D_P(z,w):={}&1+\frac{1}{(|P|^2z^d-1)(1+w^d)}
\biggl\{-w^{2d}-\frac{w^{3d}}{|P|}+\frac{w^d}{|P|^2z^d}
 +|P|(w^2z)^d \notag\\
&\hspace{35mm} +(w^2z)^d-|P|^2(wz^2)^d+(w^3z)^d
 -|P|^2(w^2z^2)^d\biggr\}.
\label{eq:DP-explicit}
\end{align}

\begin{lemma}[Square-dual Euler products]\label{lem:square-euler-products}
For $q^{-2}<|z|<q^{-1}$ and $|w|<q^{-1}$, all the series and products below converge
absolutely, and
\begin{align}
\mathcal B(z,w)
 &=\mathcal Z_{\A}(z)\mathcal Z_{\A}(w)\mathcal Z_{\A}(qw^2z)
   \prod_{P\in\Pp}B_P(z,w),\label{eq:B-factorization}\\
\prod_{P\in\Pp}B_P(z,w)
 &=\mathcal Z_{\A}\!\left(\frac{w}{q^2z}\right)
   \mathcal Z_{\A}(w^2)^{-1}\prod_{P\in\Pp}D_P(z,w).
\label{eq:BD-factorization}
\end{align}
Both residual products continue to converge absolutely beyond this region: the first one
whenever $|w|<q|z|$, $|w|<q^{-1/2}$ and $|wz|<q^{-1}$, and the second one whenever
\begin{equation}\label{eq:D-convergence-region}
 |w|^2<q|z|,\qquad |w|<q^{3}|z|^{2},\qquad |w|<1,
 \qquad |wz|<q^{-1}.
\end{equation}
\end{lemma}

\begin{proof}
The local computation is that of \cite{Fl17}*{Lemmas~6.2--6.3}: the $P$-part of
\eqref{eq:B-bivariate} is obtained by writing $a=v_P(Y)$ and $b=v_P(f)$ and inserting
\eqref{eq:local-gauss-table}, and summing the resulting geometric series gives
\eqref{eq:BP-explicit} and \eqref{eq:DP-explicit} together with the stated regions of
absolute convergence.  The one ingredient not contained in \cite{Fl17} is that these
$P$-parts multiply, which is \eqref{eq:Gc-multiplicativity-odd}; as noted after
\eqref{eq:B-bivariate}, the unnormalized sums are not multiplicative in general when
$\eta_q=-1$.
\end{proof}

The contour reduction now proceeds exactly as in \cite{Fl17}*{(6.2)--(6.5)}: the range
$d(C)\le g-1$ of \eqref{eq:SN-even-sector} is completed to all $C\mid f^\infty$,
Perron's coefficient formula is applied to the $Y$-range and to the even $f$-range, and
\eqref{eq:B-factorization}--\eqref{eq:BD-factorization} are substituted.  This gives,
uniformly for $N\in\{g,g-1\}$,
\begin{align}\label{eq:square-double-integral}
\mathcal S_N(V=\square)
=-q^{2g+1}\left(\frac1{2\pi i}\right)^2
\oint_{|w|=r}\oint_{|z|=q^{-3/2}}
&\frac{z^g(1-1/(qz))(1-qw^2)\prod_PD_P(z,w)}
{w(1-z)(1-w/(qz))(1-qw)(1-q^2w^2z)^2(q^2w^2z)^{\lfloor N/2\rfloor}}\,dz\,dw \notag\\
&+O_{q,\e}\!\left(q^{g(1+\e)/2}\right),
\end{align}
where $r<q^{-1}$ is fixed sufficiently small.  The error term absorbs the completion of
the $C$-range together with the tails of the Perron truncations, and is the error term
of \cite{Fl17}*{(6.3)--(6.4)}.
For $N=g-1$, \eqref{eq:square-double-integral} is, factor for factor, the display
following \cite{Fl17}*{(6.5)}, obtained there by substituting
\cite{Fl17}*{Lemma~6.3}; the case $N=g$ differs from it only in the exponent
$\lfloor N/2\rfloor$, which is the sole place where $N$ enters.  We record this
explicitly because it localizes the entire difference between the
present computation and that of \cite{Fl17} in the residue bookkeeping at the secondary
circle, carried out in \S\ref{subsec:three-secondary-residues}.

Move the $w$-contour in \eqref{eq:square-double-integral} outwards to
$|w|=q^{-1/4-\delta}$, with $0<\delta<1/4$.  On $|z|=q^{-3/2}$ the poles of the integrand
in $w$ lie at $w=q^{-1}$, at $w=qz$, of modulus $q^{-1/2}$, and at the two points with
$q^2w^2z=1$, of modulus $q^{-1/4}$.  The first two are crossed, and their contributions
to \eqref{eq:square-double-integral} are $A_N$ and $B_N$ below; the double poles of
$(1-q^2w^2z)^{-2}$ are not crossed, since $q^{-1/4-\delta}<q^{-1/4}$.
The integral on the new contour is $O_{q,\e}\!\left(q^{g(1+\e)/2}\right)$.  Indeed, on
that contour $|q^2w^2z|=q^{-2\delta}$, so that the inequality $\delta>0$ which keeps the
double poles outside is the same as the first inequality of
\eqref{eq:D-convergence-region}; the four inequalities there then bound
$\prod_{P}D_P(z,w)$ in terms of $q$, while $|z^g|=q^{-3g/2}$ and
$|(q^2w^2z)^{-\lfloor N/2\rfloor}|\le q^{\delta g}$, and it suffices to take
$\delta<\e/2$.
This is the estimate in \cite{Fl17}*{(6.8)--(6.9)}, and the displayed
integral shows that its constants and exponents are uniform for both
values of $N$.  Consequently
\begin{equation}\label{eq:AB-decomposition}
\mathcal S_N(V=\square)=A_N+B_N+O_{q,\e}\!\left(q^{g(1+\e)/2}\right),
\end{equation}
where
\begin{align}
A_N&=-\frac{q^{2g+1}}{2\pi i}\oint_{|z|=q^{-3/2}}
\frac{z^g\left(1-\frac1{qz}\right)\prod_{P\in\Pp}B_P(z,q^{-1})}{(1-z)^3z^{\lfloor N/2\rfloor}}\,dz,
\label{eq:A-principal}\\
B_N&=-\frac{q^{2g+1}}{2\pi i}\oint_{|z|=q^{-3/2}}
\frac{z^g\left(1-\frac1{qz}\right)(1-q^3z^2)\prod_{P\in\Pp}D_P(z,qz)}
{(1-z)(1-q^2z)(1-q^4z^3)^2(q^4z^3)^{\lfloor N/2\rfloor}}\,dz.
\label{eq:B-secondary}
\end{align}

We first identify the contribution of $A_g+A_{g-1}$.
For $w=q^{-1}$, Lemma~\ref{lem:square-euler-products} gives absolute convergence 
of $\prod_{P\in\Pp}B_P(z,q^{-1})$ throughout $q^{-2}<|z|<1$.
Hence, for any $1<\rho<q$, the contour in \eqref{eq:A-principal} may be deformed 
from $|z|=q^{-3/2}$ to $|z|=\rho^{-1}$ without crossing the pole at $z=1$.
In the scaled $u$-variable used in the present paper, the Euler-product identity immediately preceding \cite{Fl17}*{(6.10)} becomes
\begin{equation}\label{eq:B-C-identity}
\left(1-\frac{u}{q}\right)\prod_{P\in\Pp}B_P\!\left(u^{-1},q^{-1}\right)\left(1-u^{-1}\right)^{-1}
=\frac{\mathcal C(u)}{\zA(2)}.
\end{equation}
Applying \eqref{eq:B-C-identity} after the change of variables $z=u^{-1}$, 
and using $\left\lfloor g/2\right\rfloor+\left\lfloor(g-1)/2\right\rfloor=g-1$, gives
\begin{align}
A_g&=-\frac{q^{2g+1}}{\zA(2)}\frac1{2\pi i}
\oint_{|u|=\rho}\frac{\mathcal C(u)}{(1-u)^2u^{\lfloor (g-1)/2\rfloor+1}}\,du, \label{eq:Ag-R0} \\
A_{g-1}&=-\frac{q^{2g+1}}{\zA(2)}\frac1{2\pi i}
\oint_{|u|=\rho}\frac{\mathcal C(u)}{(1-u)^2u^{\lfloor g/2\rfloor+1}}\,du. \label{eq:Agminus1-R0}
\end{align}

Comparing \eqref{eq:Ag-R0}--\eqref{eq:Agminus1-R0} with \eqref{eq:R0}, we obtain
\begin{equation}\label{eq:A-R0-cancellation}
A_g+A_{g-1}=-\mathcal R_0(g;\rho).
\end{equation}
It follows from \eqref{eq:AB-decomposition} and \eqref{eq:A-R0-cancellation} that
\begin{equation*}
\mathcal R_0(g;\rho)+\mathcal S_g(V=\square)+\mathcal S_{g-1}(V=\square)
=B_g+B_{g-1}+O_{q,\e}\!\left(q^{g(1+\e)/2}\right).
\end{equation*}

Thus it remains to evaluate $B_g+B_{g-1}$.
From \eqref{eq:B-secondary}, we have
\begin{equation}\label{eq:B-sum}
B_g+B_{g-1}=-\frac{q^{2g+1}}{2\pi i}\oint_{|z|=q^{-3/2}} F(z)\,dz,
\end{equation}
where
\begin{align*}
F(z):=\frac{z^g\left(1-\frac1{qz}\right)(1-q^3z^2)\prod_{P\in\Pp}D_P(z,qz)}{(1-z)(1-q^2z)(1-q^4z^3)^2}
\left\{(q^4z^3)^{-\lfloor g/2\rfloor}+(q^4z^3)^{-\lfloor (g-1)/2\rfloor}\right\}.
\end{align*}
The Euler product $\prod_{P\in\Pp}D_P(z,qz)$ is absolutely and locally uniformly convergent 
in the annulus $q^{-2}<|z|<q^{-1}$, by \eqref{eq:D-convergence-region} with $w=qz$.  
We may therefore enlarge the contour in \eqref{eq:B-sum} from $|z|=q^{-3/2}$ to $|z|=q^{-1-\eta}$, where $0<\eta<1/3$.
The singularity at $z=0$ and the pole of $(1-q^2z)^{-1}$ at $z=q^{-2}$ lie inside the initial contour, 
while the pole at $z=1$ lies outside the final contour.  
Hence the only poles crossed in this contour deformation arise from $(1-q^4z^3)^{-2}$ 
and are $z_k=q^{-4/3}\omega^k$ ($k=0,1,2$), where $\omega=e^{2\pi i/3}$.  
We refer to $z_0,z_1,z_2$ as the secondary points; they lie on the secondary circle $|z|=q^{-4/3}$.
Hence we have
\begin{align*}
B_g+B_{g-1}=q^{2g+1}\sum_{k=0}^2\Res_{z=z_k}F(z)-\frac{q^{2g+1}}{2\pi i}\oint_{|z|=q^{-1-\eta}}F(z)\,dz.
\end{align*}
We now estimate the integral over the new contour.  
On $|z|=q^{-1-\eta}$, one has $q^4|z|^3=q^{1-3\eta}>1$.
Moreover, the Euler product $\prod_{P\in\Pp}D_P(z,qz)$ 
and all the remaining rational factors in $F(z)$ are bounded in terms of $q$ and $\eta$.  
Therefore
$$
|F(z)|\ll_{q,\eta}q^{-g(1+\eta)}q^{-(1-3\eta)\lfloor (g-1)/2\rfloor}.
$$
Since the length of the new contour is $2\pi q^{-1-\eta}$, it follows that
\begin{align*}
\frac{q^{2g+1}}{2\pi}
\left|\oint_{|z|=q^{-1-\eta}}F(z)\,dz\right|
\ll_{q,\eta} q^{g(1-\eta)} q^{-(1-3\eta)\lfloor (g-1)/2\rfloor}
\ll_{q,\eta} q^{g(1+\eta)/2}.
\end{align*}
Thus, choosing $0<\eta<\min\{\e,1/3\}$, we conclude that
\begin{equation}\label{eq:B-three-residues}
B_g+B_{g-1}=q^{2g+1}\sum_{k=0}^2\Res_{z=z_k}F(z)+O_{q,\e}\!\left(q^{g(1+\e)/2}\right).
\end{equation}

\subsection{Evaluation of the three secondary residues}
\label{subsec:three-secondary-residues}

We now evaluate the three residues in \eqref{eq:B-three-residues}.  
Put
\begin{equation}\label{eq:Phi-def}
\Phi(z):=\frac{\left(1-\frac1{qz}\right)(1-q^3z^2)\prod_{P\in\Pp}D_P(z,qz)}{(1-z)(1-q^2z)}
\end{equation}
and
$$
H_g(z):=(q^4z^3)^{-\lfloor g/2\rfloor}+(q^4z^3)^{-\lfloor (g-1)/2\rfloor}.
$$
Then
\begin{equation*}
F(z)=\frac{z^g\Phi(z)H_g(z)}{(1-q^4z^3)^2}.
\end{equation*}
By \eqref{eq:D-convergence-region} with $w=qz$, the Euler product
$\prod_{P\in\Pp}D_P(z,qz)$ is holomorphic in the annulus $q^{-2}<|z|<q^{-1}$, which
contains the secondary points $z_0,z_1,z_2$, while the remaining factors of $\Phi$ have
their poles at $z=1$ and $z=q^{-2}$.  Hence $\Phi$ is holomorphic in that annulus, and in
particular in a neighbourhood of each secondary point.
Let $z_*$ denote any one of $z_0,z_1,z_2$.
Since $q^4z_*^3=1$ and $\lfloor{g}/{2}\rfloor + \lfloor{(g-1)}/{2}\rfloor=g-1$, we have
\begin{equation}\label{eq:H-values}
H_g(z_*)=2, \quad H_g'(z_*)=-\frac{3(g-1)}{z_*}.
\end{equation}
Since $1-q^4z^3=-q^4(z-z_*)(z^2+zz_*+z_*^2)$, we obtain
\begin{align}\label{eq:double-residue-derivative}
\Res_{z=z_*}F(z)=\left. \frac{d}{dz}\left\{\frac{z^g\Phi(z)H_g(z)}{q^8(z^2+zz_*+z_*^2)^2}\right\}\right|_{z=z_*}.
\end{align}
At $z=z_*$, we have 
\begin{equation}\label{eq:double-pole-factor-values}
\left.\frac1{q^8(z^2+zz_*+z_*^2)^2}\right|_{z=z_*}=\frac{z_*^2}{9}, \quad
\left.\frac{d}{dz}\frac1{q^8(z^2+zz_*+z_*^2)^2}\right|_{z=z_*}=-\frac{2z_*}{9}.
\end{equation}
Substituting \eqref{eq:H-values} and \eqref{eq:double-pole-factor-values} 
into \eqref{eq:double-residue-derivative}, we obtain
\begin{equation}\label{eq:one-z-residue}
\Res_{z=z_*}F(z)=\frac{z_*^{g+1}}9\left(2z_*\Phi'(z_*)-(g+1)\Phi(z_*)\right).
\end{equation}
Applying \eqref{eq:one-z-residue} to $z_k$ ($k=0,1,2$) gives
\begin{align}\label{eq:three-residues-explicit}
q^{2g+1}\sum_{k=0}^2\Res_{z=z_k}F(z)
=\frac{q^{2g/3-1/3}}9\sum_{k=0}^2\omega^{k(g+1)}\left(2z_k\Phi'(z_k)-(g+1)\Phi(z_k)\right).
\end{align}

The genus enters the right-hand side of \eqref{eq:three-residues-explicit} in two ways:
through the factor $g+1$, and through the cube root of unity $\omega^{k(g+1)}$, which
depends on $g$ only through $g\bmod3$.  The latter is the source of the periodicity.
Accordingly, for $m\in\{0,1,2\}$ define
\begin{align}
\alpha_m&:=-\frac{q^{-1/3}}9\sum_{k=0}^2\omega^{k(m+1)}\Phi(z_k),\label{eq:alpha-periodic} \\
\beta_m&:=\frac{q^{-1/3}}9\sum_{k=0}^2\omega^{k(m+1)}\left(2z_k\Phi'(z_k)-\Phi(z_k)\right). \label{eq:beta-periodic}
\end{align}
Since $\Phi$ and the points $z_0,z_1,z_2$ do not involve $g$, the six numbers
$\alpha_m,\beta_m$ depend only on $q$.
Separating in \eqref{eq:three-residues-explicit} the part proportional to $g$ from the
rest, and using $\omega^{k(g+1)}=\omega^{k((g\bmod3)+1)}$, we obtain
\begin{equation}\label{eq:secondary-residue-sum}
q^{2g+1}\sum_{k=0}^2\Res_{z=z_k}F(z)
=q^{2g/3}\left(\alpha_{g\bmod3}\,g+\beta_{g\bmod3}\right)=\Sodd(g).
\end{equation}
In \cite{Fl17} only the term $k=0$ of \eqref{eq:three-residues-explicit} is retained: the
integrand of \cite{Fl17}*{(6.7)} is there described as having a double pole at
$z=q^{-4/3}$, and only the residue at that point is taken, in the display preceding
\cite{Fl17}*{(6.11)}.  The two nonreal points are compared with \cite{Fl17} in
\S\ref{subsec:Florea-comparison}.
Consequently,
\begin{equation}\label{eq:B-periodic-evaluation}
B_g+B_{g-1}=\Sodd(g)+O_{q,\e}\!\left(q^{g(1+\e)/2}\right).
\end{equation}

\subsection{Reality, exact period, and the square-dual evaluation}

We first record the nonvanishing needed below.

\begin{lemma}\label{lem:secondary-nonvanishing}
The Euler product
$$
\prod_{P\in\Pp}D_P(z,qz)
$$
is nonzero at each of the secondary points $z_0,z_1,z_2$.
Consequently, $\Phi(z_k)\ne0$ for $k=0,1,2$.
\end{lemma}

\begin{proof}
The Euler product $\prod_{P\in\Pp}D_P(z,qz)$ converges absolutely at the secondary points
by \eqref{eq:D-convergence-region} with $w=qz$.  
It remains to check that $D_P(z_k,qz_k)\ne0$ for every $P\in\Pp$ and every $k$.

Fix $P\in\Pp$, write $d=d(P)$ and put $\zeta:=z_k^{d}=|P|^{-4/3}\omega^{kd}$, so that
$\zeta^3=|P|^{-4}$.  Specializing \eqref{eq:DP-explicit} to $w=qz$ and evaluating at
$z=z_k$ gives
$$
D_P(z_k,qz_k)
=1+\frac{|P|^{-1}-|P|^2\zeta^{2}+|P|^3(1-|P|)\zeta^{4}}
        {\bigl(|P|^2\zeta-1\bigr)\bigl(1+|P|\zeta\bigr)},
$$
and the denominator is nonzero because $|P|^2|\zeta|=|P|^{2/3}\ne1$ and $|P||\zeta|=|P|^{-1/3}<1$.
Adding the two terms, the numerator of $D_P(z_k,qz_k)$ equals
$$
-\frac{|P|-1}{|P|}\left(|P|^4\zeta^{4}-|P|^3\zeta^{2}-|P|^2\zeta+1\right)
=\frac{|P|-1}{|P|}\left(|P|^3\zeta^{2}+\bigl(|P|^2-1\bigr)\zeta-1\right),
$$
the second equality because $|P|^4\zeta^{4}=|P|^4\zeta\cdot\zeta^3=\zeta$.  Since
$|P|>1$, the vanishing of $D_P(z_k,qz_k)$ therefore forces $|P|^3\zeta^{2}+\bigl(|P|^2-1\bigr)\zeta-1=0$.
The two roots of this quadratic are real, its discriminant
$\bigl(|P|^2-1\bigr)^2+4|P|^3$ being positive.
However $\zeta=|P|^{-4/3}\omega^{kd}$ is either nonreal or equal to $|P|^{-4/3}$, and in
the latter case the left-hand side equals $|P|^{1/3}+|P|^{2/3}-|P|^{-4/3}-1>0$, 
since $|P|>1$.  Thus $D_P(z_k,qz_k)\ne0$ for every $P\in\Pp$, and, the Euler product
being absolutely convergent at $z_k$, it follows that
$\prod_{P\in\Pp}D_P(z_k,qz_k)\ne0$.
The remaining factors in \eqref{eq:Phi-def} are also nonzero at $z=z_k$, and hence
$\Phi(z_k)\ne0$.
\end{proof}

\begin{proposition}\label{prop:reality-period}
The coefficients $\alpha_m$ and $\beta_m$ are real for each $m\in\{0,1,2\}$.  Moreover, if
$(\alpha_m,\beta_m)$ is extended periodically from $m\in\{0,1,2\}$ to $m\in\mathbb Z$,
the resulting sequence has minimal period exactly $3$.
\end{proposition}

\begin{proof}
We have $z_0=q^{-4/3}\in\mathbb R$ and $z_2=\overline{z_1}$.  The local factors
$D_P(z,qz)$ and the rational factor in \eqref{eq:Phi-def} have real coefficients, so
$\Phi$ is real on the real points of the annulus $q^{-2}<|z|<q^{-1}$ and therefore
satisfies $\Phi(\bar z)=\overline{\Phi(z)}$ there, as does $\Phi'$.  Since
$\omega^{2(m+1)}=\overline{\omega^{m+1}}$, the terms $k=1$ and $k=2$ in
\eqref{eq:alpha-periodic} and \eqref{eq:beta-periodic} are complex conjugates of one
another, while the terms with $k=0$ are real.  Hence $\alpha_m$ and $\beta_m$ are real.

The sequence $(\alpha_m,\beta_m)_{m\in\mathbb Z}$ is $3$-periodic by construction, so its
minimal period is $1$ or $3$, and it remains only to rule out $1$.
Applying the discrete Fourier transform to \eqref{eq:alpha-periodic}, and using that
$\sum_{m=0}^2\omega^{m(k-1)}$ equals $3$ if $k=1$ and $0$ if $k\in\{0,2\}$, we obtain
\begin{align}\label{eq:alpha-fourier}
\sum_{m=0}^2\alpha_m\omega^{-m}
=-\frac{q^{-1/3}}9\sum_{k=0}^2\omega^k\Phi(z_k)\sum_{m=0}^2\omega^{m(k-1)}
=-\frac{q^{-1/3}}3\,\omega\,\Phi(z_1).
\end{align}
By Lemma~\ref{lem:secondary-nonvanishing}, $\Phi(z_1)\ne0$, so the Fourier coefficient in
\eqref{eq:alpha-fourier} does not vanish and $m\mapsto\alpha_m$ is not constant.
Therefore the minimal period is exactly $3$.
\end{proof}

The reduction of \S\ref{subsec:secondary-reduction} and the residue evaluation
of \S\ref{subsec:three-secondary-residues} now combine to prove the square-dual
evaluation stated in \S\ref{subsec:square-dual}.

\begin{proof}[Proof of Proposition~\ref{prop:square-dual-evaluation}]
Let $\e>0$ and let $\rho$ satisfy $1<\rho<q$.  Summing \eqref{eq:AB-decomposition} over
$N\in\{g,g-1\}$ gives
$$
\mathcal S_g(V=\square)+\mathcal S_{g-1}(V=\square)
=(A_g+A_{g-1})+(B_g+B_{g-1})+O_{q,\e}\!\left(q^{g(1+\e)/2}\right).
$$
The first bracket equals $-\mathcal R_0(g;\rho)$ by \eqref{eq:A-R0-cancellation}, and the
second equals $\Sodd(g)+O_{q,\e}\bigl(q^{g(1+\e)/2}\bigr)$ by
\eqref{eq:B-periodic-evaluation}.  Adding $\mathcal R_0(g;\rho)$ to both sides therefore
gives
\begin{align*}
\mathcal R_0(g;\rho)+\mathcal S_g(V=\square)+\mathcal S_{g-1}(V=\square)
=\Sodd(g)+O_{q,\e}\!\left(q^{g(1+\e)/2}\right),
\end{align*}
which is the assertion.  The coefficients occurring in $\Sodd(g)$ are those defined in
\eqref{eq:alpha-periodic}--\eqref{eq:beta-periodic}; they depend only on $q$, and by
Proposition~\ref{prop:reality-period} they are real and the extended sequence
$(\alpha_m,\beta_m)_{m\in\mathbb Z}$ has minimal period exactly $3$, as asserted in
Theorem~\ref{thm:main}.
\end{proof}

\subsection{Comparison with Florea}\label{subsec:Florea-comparison}

This subsection does two things: it identifies the positive-real-pole contribution
with the polynomial obtained by Florea, and it shows that the nonreal poles are
genuinely present.  Only the first requires a restriction on $q$, and that
restriction is imposed inside Remark~\ref{rem:Florea-identification} alone;
everything else here is valid for every odd prime power $q$.

Recall that $z_0=q^{-4/3}$ is the positive real secondary pole, 
while $z_1$ and $z_2=\overline{z_1}$ are the two nonreal conjugate secondary poles.
By \eqref{eq:three-residues-explicit}, the contribution of the positive real pole $z_0$ is
\begin{equation*}
\mathcal S_+(g):=\frac{q^{2g/3-1/3}}9\left(2z_0\Phi'(z_0)-(g+1)\Phi(z_0)\right).
\end{equation*}
Define the linear polynomial
\begin{equation*}
R_+(X):=\frac{q^{-2/3}}{18}\left(4z_0\Phi'(z_0)-(X+1)\Phi(z_0)\right).
\end{equation*}
Then
\begin{equation*}
\mathcal S_+(g)=q^{(2g+1)/3}R_+(2g+1).
\end{equation*}

Since $z_2=\overline{z_1}$ and $\Phi(z_2)=\overline{\Phi(z_1)}$, the contribution of the two nonreal poles is
\begin{equation}\label{eq:S-nonreal}
\mathcal S_{\mathrm{nr}}(g)
:=\frac{2q^{2g/3-1/3}}9\operatorname{Re}\!\left[\omega^{g+1}\left(2z_1\Phi'(z_1)-(g+1)\Phi(z_1)\right)\right].
\end{equation}
By \eqref{eq:secondary-residue-sum}, we obtain
\begin{equation}\label{eq:secondary-splitting}
\Sodd(g)=\mathcal S_+(g)+\mathcal S_{\mathrm{nr}}(g).
\end{equation}

The two summands $\mathcal S_+(g)$ and $\mathcal S_{\mathrm{nr}}(g)$ are distinguished by
their behaviour across residue classes.  Summing \eqref{eq:alpha-periodic} and
\eqref{eq:beta-periodic} over $m\in\{0,1,2\}$, and using that
$\sum_{m=0}^2\omega^{k(m+1)}$ equals $3$ if $k=0$ and $0$ if $k\in\{1,2\}$, gives
\begin{equation}\label{eq:mean-over-classes}
\frac13\sum_{m=0}^2\alpha_m=-\frac{q^{-1/3}}9\Phi(z_0),\qquad
\frac13\sum_{m=0}^2\beta_m=\frac{q^{-1/3}}9\left(2z_0\Phi'(z_0)-\Phi(z_0)\right).
\end{equation}
Comparing with the definition of $\mathcal S_+$, we see that $\mathcal S_+(g)$ is the
average over $m\in\{0,1,2\}$ of the three quantities $q^{2g/3}(\alpha_mg+\beta_m)$, of
which $\Sodd(g)$ is the one with $m\equiv g\bmod3$; and $\mathcal S_{\mathrm{nr}}(g)$ is the
deviation of that one from the average.  The secondary term of
\cite{Fl17}*{Theorem~1.1} is therefore the mean of the three, and not any one of them.
The three need not be close to that mean: for $q=3$ the coefficients
$\alpha_0,\alpha_1,\alpha_2$ computed in \S\ref{sec:numerical} are approximately $7.49$,
$-7.99$ and $3.50$ times their mean, and one of them has the opposite sign.

We first identify the positive-real-pole contribution with Florea's polynomial.
This is a matter of matching two displayed formulas rather than a new assertion,
so we record it as a remark.  Since \cite{Fl17}*{Theorem~1.1} is stated for $q$ prime
with $q\equiv1\bmod4$, the identification is made under that hypothesis; in that range
the quadratic Gauss phase is trivial and our square-dual normalization coincides with
that of \cite{Fl17}.

\begin{remark}\label{rem:Florea-identification}
Assume that $q$ is prime and $q\equiv1\bmod4$, as in \cite{Fl17}*{Theorem~1.1}.  The
polynomial $R$ of that theorem is obtained in \cite{Fl17}*{Section~6} from the residue of
\cite{Fl17}*{(6.7)} at the positive real double pole $z_0=q^{-4/3}$: that residue gives
the linear terms $P_1$ and $P_2$ of \cite{Fl17}*{(6.11)}, whose sum is written in
\cite{Fl17}*{(6.12)} as $q^{(2g+1)/3}R(2g+1)$.  Under the present normalization the
one-variable secondary integral \eqref{eq:B-secondary} is that same integral, since the
local factor $D_P(z,qz)$ is precisely the factor of \cite{Fl17}*{Lemma~6.3}.  Hence
$R=R_+$, and \eqref{eq:secondary-splitting} reads
\begin{equation}\label{eq:Florea-comparison}
\Sodd(g)=q^{(2g+1)/3}R_+(2g+1)+\mathcal S_{\mathrm{nr}}(g).
\end{equation}
We compute $R_+$ explicitly, in the shape of \cite{Fl17}*{(6.13)}, so that the two
displays may be compared term by term.

Put $\mathcal D(z)=\prod_PD_P(z,qz)$ and
\begin{equation*}
 C_0:=\frac{\zA(5/3)\zA(7/3)}{\zA(4/3)^2},\qquad z_0=q^{-4/3}.
\end{equation*}
Straight substitution in the rational factor in \eqref{eq:Phi-def} gives
\begin{align}
 \Phi(z_0)&=-C_0\mathcal D(z_0),\label{eq:Phi-z0-Florea}\\
 \frac12-2z_0\left.\frac{d}{dz}\log\!\left(
 \frac{(1-1/(qz))(1-q^3z^2)}{(1-z)(1-q^2z)}\right)\right|_{z=z_0}
 &=\frac{\zA(7/3)}{q^{4/3}}
 \left(-\frac52-2q^{1/3}-2q+\frac{q^{4/3}}2\right).
\label{eq:rational-log-Florea}
\end{align}
Inserting \eqref{eq:Phi-z0-Florea}--\eqref{eq:rational-log-Florea} into the definition of
$R_+$ yields
\begin{align}
R_+(X)={}&\frac{\zA(5/3)\zA(7/3)}
 {9q^{2/3}\zA(4/3)^2}\mathcal D(z_0)
 \biggl\{\frac X2+\frac{\zA(7/3)}{q^{4/3}}
 \left(-\frac52-2q^{1/3}-2q+\frac{q^{4/3}}2\right)
 -2z_0\frac{\mathcal D'(z_0)}{\mathcal D(z_0)}\biggr\},
\label{eq:RF-direct}
\end{align}
which is well defined because $\mathcal D$ is holomorphic and nonvanishing at $z_0$ by
Lemma~\ref{lem:secondary-nonvanishing}.

The factor preceding the brace in \eqref{eq:RF-direct} and the first two terms inside the
brace agree, term for term, with those of \cite{Fl17}*{(6.13)}.  The third term is printed there with
one factor $q^{-4/3}$ more than in \eqref{eq:RF-direct}; differentiating the double pole
of \cite{Fl17}*{(6.7)} produces the single factor $z_0=q^{-4/3}$ displayed above.  We
therefore compare with the residue itself, and fix \eqref{eq:RF-direct} once and for all
as its residue-normalized form; the explicit evaluation of
$z_0\mathcal D'(z_0)/\mathcal D(z_0)$ printed after \cite{Fl17}*{(6.13)} is not used
here.
\end{remark}

The substantive assertion of this subsection is the following lower bound, which
shows that the nonreal poles are genuinely present.  It carries no restriction on $q$.

\begin{proposition}\label{prop:nonreal-lower-bound}
Let $q$ be any fixed odd prime power.
There are at least two residue classes $m\bmod3$ for which there is a constant
$c_{q,m}>0$ with
$$
|\mathcal S_{\mathrm{nr}}(g)| \ge c_{q,m}\,g\,q^{2g/3}
$$
for all sufficiently large $g\equiv m\bmod 3$.  In particular $\mathcal S_{\mathrm{nr}}$
is not identically zero.
\end{proposition}

\begin{proof}
By Lemma~\ref{lem:secondary-nonvanishing}, $\Phi(z_1)\ne0$.  Write
$\Phi(z_1)=re^{i\theta}$ with $r>0$.  Then
$\operatorname{Re}\left(\omega^{m+1}\Phi(z_1)\right)=r\cos\!\left(\theta+\tfrac{2\pi(m+1)}3\right)$,
and the three angles $\theta+\tfrac{2\pi(m+1)}3$, $m\in\{0,1,2\}$, are pairwise
distinct modulo $2\pi$ and differ by $2\pi/3$.  Since the cosine vanishes only at
the two points $\pi/2$ and $3\pi/2$ modulo $2\pi$, which differ by $\pi$, at most one
of the three values can vanish.  Hence the set $M$ of those $m\in\{0,1,2\}$ with
$\operatorname{Re}\left(\omega^{m+1}\Phi(z_1)\right)\ne0$ has at least two elements.
For any $m\in\{0,1,2\}$ and any $g\equiv m\bmod 3$, \eqref{eq:S-nonreal} gives
\begin{align*}
\mathcal S_{\mathrm{nr}}(g)
=-\frac{2q^{-1/3}}9\operatorname{Re}\left(\omega^{m+1}\Phi(z_1)\right) gq^{2g/3} +O_q(q^{2g/3}).
\end{align*}
For $m\in M$ the coefficient of $gq^{2g/3}$ is nonzero, which proves the stated
lower bound.
\end{proof}

\begin{remark}
The three zeros of $1-q^4z^3$ have the same modulus.
Hence a circular contour deformation across the secondary circle crosses all three
double poles simultaneously, and this is so for either direction of the deformation:
the outward shift from $|z|=q^{-3/2}$ used in \cite{Fl17}*{Section~6} and the inward
shift to $|z|=q^{-3/2-\e}$ indicated in \cite{Fl17}*{Remark~3} both sweep the entire
circle $|z|=q^{-4/3}$.
Remark~\ref{rem:Florea-identification} shows, in \eqref{eq:Florea-comparison}, that
Florea's polynomial is precisely the contribution of the positive real pole, whereas by
Proposition~\ref{prop:nonreal-lower-bound} the two nonreal conjugate poles contribute
at the same order and produce the exact period-three dependence.
\end{remark}

\section{Numerical evidence for period three}\label{sec:numerical}

The purpose of this section is not to test the extension to arbitrary odd prime powers,
but to display the period-three dependence in actual finite-ensemble moments.
We therefore fix $q=3$ and compute enough genera to obtain four points in each
residue class modulo $3$.
No numerical input is used in the proof of Theorem~\ref{thm:main}.

Put
$$
Y_g=3^{-2g/3}\bigl(M_1(g)-\mathcal Q_g\bigr),
$$
where $\mathcal Q_g$ is the principal term \eqref{eq:principal-term}.
Theorem~\ref{thm:main} predicts that, for the unique $m\in\{0,1,2\}$ satisfying
$m\equiv g\bmod 3$,
\begin{equation}\label{eq:Y-prediction}
Y_g=\alpha_mg+\beta_m+O_\e\!\left(3^{-g(1/6-\e/2)}\right).
\end{equation}
Thus the numerical signature is not one linear trend with a small oscillatory error.
Rather, the sequence should split into three separate linear trends indexed by
$g\bmod3$.

The exact moments were obtained from the coefficient identity
\begin{equation*}
\sum_{D\in\Hh_{2g+1}}\chi_D(f)
=[u^{2g+1}]\left(\mathcal L(u,\psi_f)(1-3u^2)
\prod_{P\mid f}(1-u^{2d(P)})^{-1}\right),
\end{equation*}
which follows in one line from
$\prod_{P\nmid f}(1+\psi_f(P)u^{d(P)})
=\mathcal L(u,\psi_f)\mathcal Z_{\A}(u^2)^{-1}\prod_{P\mid f}(1-u^{2d(P)})^{-1}$.
The only nontrivial input is therefore the vector of coefficients
$\sigma_m(f)=\sum_{H\in\A_m^+}\psi_f(H)$ for $m<d(f)$.
Since $\psi_f$ is completely multiplicative in its argument, these are obtained by a
smallest-prime-factor sieve over $\A^+_{<d(f)}$ from the values $\psi_f(P)$,
and each $\psi_f(P)$ is evaluated by quadratic reciprocity as the quadratic
character of $\A/(P)$ at $f\bmod P$.  The accompanying code implements this scheme
together with the reductions described next, and reaches $g=15$.

A scaling symmetry reduces the computation further.  
Let $c\in\mathbb F_3^\times$ be the nonsquare element and put
$$
A^{(c)}(T):=c^{-d(A)}A(cT), \qquad (A\in\A).
$$
This transformation preserves monicity, degree, and square-freeness, while
$$
\chi_{D^{(c)}}(f^{(c)})
=\chi_2(c)^{d(D)d(f)}\chi_D(f).
$$
Since $d(D)=2g+1$ is odd, the aggregate contribution of every odd $d(f)$ vanishes.
For nonsquare $f$ of even degree, the character $\psi_f$ is nonprincipal and even.
Hence $\mathcal L(u,\psi_f)$ has the trivial factor $1-u$, so
$\mathcal L(1,\psi_f)=0$; the top coefficient of $\mathcal L(u,\psi_f)$ is therefore
recovered from the lower coefficients.
Together with a direct treatment of the perfect-square case, this gives exact values
of $M_1(g)$ for $1\le g\le15$.

All coefficient extractions were carried out in exact integer or rational arithmetic.
In the accompanying code, each vector dot product is preceded by the
rigorous bound
$r\max_i|a_i|\max_i|b_i|\le 2^{63}-1$ (where $r$ is the vector length);
when this fails, accumulation is automatically performed with Python
integers.  Thus the use of NumPy arrays cannot introduce silent signed
$64$-bit wraparound.

As direct consistency checks, the first four values are
$$
M_1(1)=36,\qquad M_1(2)=448,\qquad M_1(3)=5120,\qquad
M_1(4)=\frac{166960}{3}.
$$
At the upper end of the computation, one obtains
$$
M_1(14)=\frac{128791663628783456}{243},\qquad
M_1(15)=\frac{410863786895242016}{81}.
$$

For the three theoretical lines, the coefficients are
\begin{center}
\renewcommand{\arraystretch}{1.15}
\begin{tabular}{c|rr}
$m$ & $\alpha_m$ & $\beta_m$\\\hline
$0$ & $\phantom{-}0.0604838375$ & $\phantom{-}0.3617706494$\\
$1$ & $-0.0645165657$ & $-0.2128045840$\\
$2$ & $\phantom{-}0.0282448827$ & $-0.0467797514$
\end{tabular}
\end{center}
Their averages are $\tfrac13\sum_m\alpha_m=0.0080707181$ and
$\tfrac13\sum_m\beta_m=0.0340621047$, which by \eqref{eq:mean-over-classes} are the
coefficients of the single line predicted by \cite{Fl17}*{Theorem~1.1}.
The six coefficients were computed by grouping the Euler products by prime degree,
using
\begin{equation*}
 \pi_3(n)=\frac1n\sum_{d\mid n}\mu(d)3^{n/d},
\end{equation*}
and truncating after prime degree $200$, the arithmetic being carried out at
$100$-decimal working precision.  At the secondary points one has
$D_P(z_k,qz_k)-1\ll|P|^{-4/3}$, so that $\log D_P(z_k,qz_k)\ll|P|^{-4/3}$ and the tail
beyond prime degree $D$ contributes $O(3^{-D/3}/D)$, of size about $10^{-34}$ at
$D=200$.  Since $|P|^{-4/3}$ falls below the working precision before prime degree
$200$, the logarithms are accumulated from $D_P(z_k,qz_k)-1$ in closed form rather than
by forming $D_P(z_k,qz_k)$ and taking its logarithm; the accuracy is then limited by the
truncation.  Repeating the calculation at degree $160$ and $80$-decimal precision leaves
all displayed digits unchanged.  The accompanying files
\path{period_three_q3_reproduce.py} and \path{M1_exact_15.json} are available as
ancillary files with the arXiv version of this paper.  The former computes
$\mathscr P(1)$, $\mathscr P'(1)/((\log3)\mathscr P(1))$, the six coefficients, the
complete CSV table, and the PDF figure from the latter, which retains the exact
numerator and denominator of every moment used here.
We plot the balanced range $3\le g\le14$, in which each residue class occurs
exactly four times:
$$
\begin{aligned}
g\equiv0\bmod 3 &: 3,6,9,12,\\
g\equiv1\bmod 3 &: 4,7,10,13,\\
g\equiv2\bmod 3 &: 5,8,11,14.
\end{aligned}
$$
For reference, the residuals
$Y_g-(\alpha_mg+\beta_m)$ at $g=12,13,14$ are approximately
$-0.0307965$, $0.0208575$, and $0.0299341$, respectively.
The separately computed value at $g=15$, omitted from the figure only to keep the
three residue classes balanced, has residual approximately $0.0437558$.
These values are consistent with the error scale in \eqref{eq:Y-prediction}.

Figure~\ref{fig:period-three} displays the outcome.
The normalized exact moments do not lie near a single linear function of $g$.
After separation by residue class, they track the three different lines in
\eqref{eq:Y-prediction}.  The separation between the three progressions persists
through the extended range, supporting the conclusion that the secondary contribution
depends genuinely on $g\bmod3$ and is not exhausted by the positive real pole.

\begin{figure}[ht]
\centering
\includegraphics[width=0.72\textwidth]{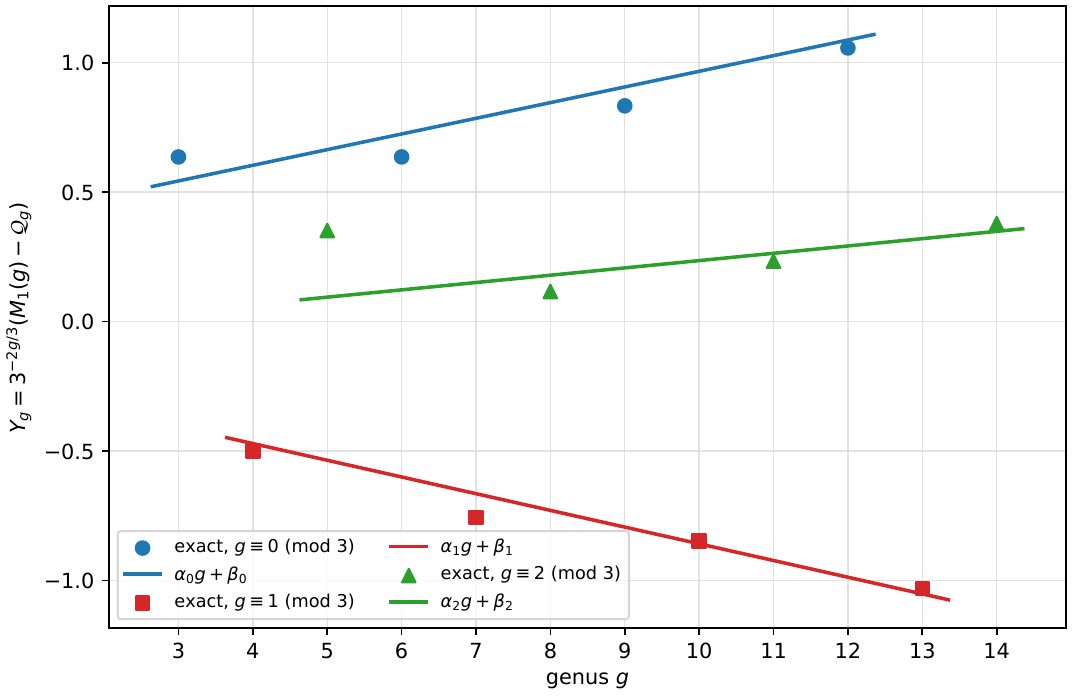}
\caption{For $q=3$, the markers are normalized exact-minus-principal moments and
the lines are $y=\alpha_mg+\beta_m$, $m=0,1,2$.
The balanced range $3\le g\le14$ contains four genera in each residue class modulo $3$,
and the exact data follow the three trends predicted by the secondary poles.}
\label{fig:period-three}
\end{figure}

\FloatBarrier

\end{document}